\documentclass[11pt,oneside,english,reqno]{amsart}

\usepackage{algorithm}
\usepackage{algorithmic}
\usepackage{multirow}

\usepackage[a4paper,left=3.5cm, right=3cm,top=3cm, bottom=3cm]{geometry}

\usepackage{caption}
\usepackage{subcaption}
\usepackage{color}
\usepackage{enumitem}

\usepackage[T1]{fontenc}
\usepackage[latin9]{inputenc}
\usepackage{babel}
\usepackage{amsmath}
\usepackage{amsthm}
\usepackage{amssymb}

\usepackage{comment}

\usepackage{graphicx}

\usepackage{graphicx}
\usepackage{tikz}
\usetikzlibrary{shapes.misc}
\usepackage{float}

\usepackage[unicode=true,pdfusetitle,
bookmarks=true,bookmarksnumbered=false,bookmarksopen=false,
breaklinks=false,pdfborder={0 0 1},backref=false,
hidelinks]
{hyperref}

\makeatletter
\numberwithin{equation}{section}
\numberwithin{figure}{section}

\makeatother

\newtheorem{theorem}{Theorem}[section]

\newtheorem{corollary}[theorem]{Corollary}

\newtheorem{definition}[theorem]{Definition}
\newtheorem{example}[theorem]{Example}

\newtheorem{lemma}{Lemma}[section]

\newtheorem{proposition}[theorem]{Proposition}
\newtheorem{remark}[theorem]{Remark}

\newcommand{\hilbertX}{{\texorpdfstring{\mathcal X}{X}}}
\newcommand{\hilbertH}{{\texorpdfstring{\mathcal H}{H}}}
\newcommand{\hilbertG}{{\texorpdfstring{\mathcal G}{G}}}

\newcommand{\multifC}{{\mathbb C}}
\newcommand{\opert}{{\mathbb T}}
\newcommand{\setD}{{\texorpdfstring{\mathcal D}{D}}}
\newcommand{\hatD}{{\texorpdfstring{\hat{\mathcal D}}{D}}}

\begin{document}
	
\title[Dynamical system related to splitting projection methods]{Dynamical system related to primal-dual splitting projection methods
}

	\subjclass[2020]{34A12, 37N40, 47J25, 34A45, 49J53}
	
\keywords{autonomous ordinary differential equations, locally Lipschitz vector field, existence and uniqueness of solutions, extendability of solutions, projected dynamical systems}
\author{
	Ewa M. Bednarczuk\texorpdfstring{$^{1,4}$}{firstsec}
}
\author{
	Raj Narayan Dhara\texorpdfstring{$^{2,5}$}{thirdfive}
}
\author{
	Krzysztof E. Rutkowski\texorpdfstring{$^{3,4}$}{fourth} 
}
\thanks{$^1$ Warsaw University of Technology, 00-662 Warszawa, Koszykowa 75,
}

\thanks{$^2$Department of Mathematics and Statistics, Masaryk University, Kotl\'a{}\v{r}sk\'a 2, 611 37 Brno, Czech Republic}

\thanks{$^3$ Cardinal Stefan Wyszy\'nski University, 01-815 Warsaw, Dewajtis 5}
\thanks{$^4$ Systems Research Institute of the Polish Academy of Sciences, Newelska 6 
}
\thanks{$^5$RND acknowledges the support of IBS PAN, Warszawa, Poland and the Czech Science Foundation, project GJ19-14413Y}

\begin{abstract}
We introduce a dynamical system to the problem of 
finding zeros of the sum of two maximally monotone operators. We investigate the existence, uniqueness and  extendability of solutions to this dynamical system in a Hilbert space. 
We prove that the trajectories of the proposed dynamical system converge strongly to a primal-dual solution of the considered problem. 
Under explicit time discretization of the dynamical system we obtain the best approximation algorithm for solving coupled monotone inclusion problem. 
\end{abstract}
\maketitle

\section{Introduction}

Let $\hilbertH$, $\hilbertG$ 
be  Hilbert spaces. We consider the problem of 
finding $p\in \hilbertH$ such that
\begin{align}\tag{P}\label{problem:composition_primal}
& 0 \in Ap+L^*BLp, 
\end{align}
where $ A: \hilbertH \rightarrow \hilbertH, 
\ B:\ \hilbertG \rightarrow \hilbertG$ are
maximally monotone operators,
$L: \hilbertH \rightarrow \hilbertG$ is a bounded, linear operator.
Together with problem \eqref{problem:composition_primal} we consider the dual problem formulated as finding $v^*\in \hilbertG$ such that 
\begin{equation}\tag{D}\label{problem:composition_dual}
0 \in -LA^{-1}(-L^*v^*)+B^{-1}v^*.
\end{equation}
To problems \eqref{problem:composition_primal} and \eqref{problem:composition_dual} we associate Kuhn-Tucker set defined as
\begin{equation}
\label{setZ}
\tag{Z}
Z:= \{ (p,v^*) \in \hilbertH \times \hilbertG \mid -L^* v^* \in Ap \quad \text{and}\quad Lp\in B^{-1}v^* \}.
\end{equation}

The set $Z$ is nonempty if and only if there exists a solution of the primal problem \eqref{problem:composition_primal} and to the dual problem \eqref{problem:composition_dual} (see \cite[Corollary 2.12]{pennanen1999dualization}).

Our aim in this paper is to investigate, for a given  $x_{0}, \bar{w}\in\hilbertH\times\hilbertG$,  the following dynamical system,  solution of which  asymptotically approaches  solution of \eqref{problem:composition_primal}-\eqref{problem:composition_dual},
\begin{align} \tag{S}
\begin{aligned}
\label{system:Haugazeau}
&\dot{x}(t)=Q(\bar{w},x(t),\opert x(t))-x(t),\quad t\geq 0,\\
&x(0)=x_0,
\end{aligned}
\end{align}
where $\opert :\ \hilbertH\times \hilbertG\rightarrow \hilbertH\times \hilbertG$, 
fixed point set of the operator $\opert$ is $Z$ ($\operatorname{Fix}\opert = Z$), with $Z$ defined by \eqref{setZ} 
and
$Q:(\hilbertH\times \hilbertG)^{3}\to\hilbertH\times \hilbertG$,
\begin{equation}\label{eq:Haugazeau_projection}
Q(\bar{w},b,c):=P_{H(\bar{w},b)\cap H(b,c)}(\bar{w}),
\end{equation}
is the projection $P$ of the element $\bar{w}$ onto the set $H(\bar{w},b)\cap H(b,c)$ which is the intersection of two  hyperplanes of the form
\begin{align}\label{Qblock}
\begin{aligned}
& H(z_1,z_2):=\{ h\in \hilbertH\times \hilbertG \mid \langle h-z_2 \mid z_1-z_2 \rangle \leq 0  \},\quad z_1,z_2 \in \hilbertH\times \hilbertG.
\end{aligned}
\end{align}
In particular,
$
H(\bar{w},b)=\{ h\in \hilbertH\times \hilbertG \mid \langle h-b \mid \bar{w}-b \rangle \leq 0  \}.
$

Under explicit discretization with step size equal to one the system~\eqref{system:Haugazeau} becomes the 
best approximation algorithm  for finding fixed point of $\opert$ introduced in
\cite[Proposition 2.1]{Alotaibi_2015_best} (see also
\cite[Theorem 30.8]{MR3616647}),
\begin{equation}\label{scheme:haugazeau}
x_{n+1}=Q(\bar{w},x_n, x_{n+1/2}),\quad n\in \mathbb{N}.
\end{equation}
with the choice of $x_{n+1/2}:=\opert(x_n)$ and the  starting point $x_0$. The characteristic feature of this algorithm is the strong convergence of the sequence $x_n$ to a fixed point of $\opert$ (see also \cite{weak_to_strong_2001})  
In contrast to this, a dynamical system investigated, e.g. in \cite{Bot2020}, is related to other primal-dual method which  exhibits weak convergence.


In case when $A=\partial f$, $B=\partial g$, $f:\ \hilbertH\rightarrow \mathbb{R}\cup \{+\infty\}$, $g:\ \hilbertH\rightarrow \mathbb{R}\cup \{+\infty\}$ are proper convex, lower semicontinuous (l.s.c.) functions, the problem \eqref{problem:composition_primal} (if solvable) reduces to finding a point $p\in \hilbertH$ 
solving the following minimization problem (see \cite{pennanen_dualization})
\begin{equation}
\text{minimize}_{p\in \hilbertH} f(p)+g(Lp)
\end{equation}
and \eqref{problem:composition_dual} reduces to finding a point $v^{*}\in \hilbertG$ 
solving the following maximization problem
\begin{equation}
\text{maximize}_{v^* \in \hilbertG} -f^*(-L^*v^*)-g^*(v^*).
\end{equation}

First order dynamical systems related to optimization problems have been discussed by many authors (see, e.g., \cite{MR3193795,MR1347800,MR3612970,MR3705347,MR2028993}). In those papers, a natural assumption is that the vector field $F$ is globally Lipschitz and consequently,  the existence and uniqueness of solutions to the dynamical system is guaranteed by classical results (see e.g. \cite[Theorem 7.3]{brezis2011}).
For instance, Abbas, Attouch and Svaiter considered the following system in~\cite{MR3193795}
\begin{align}\label{system:Abbas}
\begin{aligned}
&\dot{x}(t)+x(t)=\operatorname{prox}_{\mu \Phi} (x(t)-\mu B(x(t))),\\
&x(0)=x_0 ,
\end{aligned}
\end{align}
where $\Phi:\ \hilbertH \rightarrow\mathbb{R}\cup \{+\infty\}$  is a proper, convex and l.s.c. function defined on a Hilbert space $\hilbertH$, $B:\ \hilbertH \rightarrow \hilbertH$ is $\beta$-cocoercive operator and $\operatorname{prox}_{\mu \Phi}:\ \hilbertH \rightarrow \hilbertH$ is a proximal operator defined as
\begin{equation*}
\operatorname{prox}_{\mu \Phi}(x)=\arg\min_{y\in \hilbertH} \{ \Phi(y)+\frac{1}{2\mu}\|x-y\|^2 \}.
\end{equation*}

\noindent
Furthermore, Bo\c{t} and Csetnek, in~\cite{MR3612970}, studied the dynamical  system 
\begin{align}\label{system:Fixed}
\begin{aligned}
&\dot{x}(t)=\lambda(t)(T(x(t))-x(t)),\quad t\geq 0\\
& x(0)=x_0,
\end{aligned}
\end{align}
where $T:\ \hilbertH \rightarrow \hilbertH$ is a nonexpansive operator, $\lambda: [0,\infty) \rightarrow [0,1]$ is a Lebesgue measurable function. By applying in \eqref{system:Fixed} operator $T$ defined as
$T=J_{\gamma A}( Id- \gamma B )$, where $A:\ \hilbertH \rightarrow \hilbertH$ is a maximally monotone operator, the system~\eqref{system:Fixed},
under special discretization (see e.g.~\cite[Remark 8]{MR3612970}), 
leads to the forward-backward algorithm for solving operator inclusion problems in the form
\begin{equation*}
\text{find}\ x\in \hilbertH \quad s.t. \quad  0\in A(x)+B(x).
\end{equation*}

For other discretizations see e.g., \cite[Section 2.3]{MR2731260}.


The most essential difference  between \eqref{system:Haugazeau} and the systems \eqref{system:Abbas}, \eqref{system:Fixed}
is that, in general, one cannot expect that the vector field $Q$ given in \eqref{system:Haugazeau} is globally Lipschitz with  respect to variable $x$ as it is  the case of dynamical systems \eqref{system:Abbas} and \eqref{system:Fixed}.


The contribution of the present investigation is as follows.
We formulate the problem and provide preliminary facts in Sections \ref{section:formulation} and \ref{section:preliminaries}, respectively. 
In Section \ref{section:existenceanduniqueness} we prove the existence and uniqueness of solutions to dynamical system 
\eqref{system:Haugazeau} by studying a more general problem  \eqref{dynamics:DS}. 
Extendability of solutions to  dynamical system 
\eqref{dynamics:DS} is studied in Section \ref{section:extendability}.
The behaviour  at $+\infty$ of solutions to \eqref{dynamics:DS} is investigated in Section \ref{section:behaviour}.
In Section \ref{section:PDS} we present applications of the results obtained  for \eqref{dynamics:DS}  
to projected dynamical systems \eqref{system:PDS}. 

\section{Formulation of the problem}\label{section:formulation}
Suppose that the set $Z$ given by \eqref{setZ} is nonempty. Then for all $x\in \hilbertH\times \hilbertG$, $Z\subset  H(x,\opert x)$. 
Let $\bar{w}\in \hilbertH\times \hilbertG$ and $\bar{z}=P_{Z}(\bar{w})$. Let us define an open ball in Hilbert space $\hilbertH\times \hilbertG$ centered at $a\in \hilbertH\times \hilbertG$ with some radius $R>0$ as follows:
\begin{equation*}
\mathbb{B}(a,R):=\{ x\in \hilbertH\times \hilbertG \mid \|a-x\|<R\}
\end{equation*}
and its closure by
\begin{equation*}
\bar{\mathbb{B}}(a,R):=\{ x\in \hilbertH\times \hilbertG \mid \|a-x\|\leq R\}.
\end{equation*}

We limit ourselves to a closed subset $\setD\subset \hilbertH\times \hilbertG$ such that for all $x \in \setD$ we have $\bar{z}\in H(\bar{w},x)$. This latter conditions ensures that $\bar{z}$ is an equilibrium point of 
$$Q(\bar{w},\cdot, \opert(\cdot)): 
\setD\to \setD.
$$

\noindent
The fact that
\begin{equation}\label{fact:subset_ball}
x \in \bar{\mathbb{B}}\left(\frac{\bar{w}+\bar{z}}{2},\frac{\|\bar{w}-\bar{z}\|}{2}\right) \text{ if and only if   }
\langle  \bar{z} - x \mid \bar{w}-x \rangle \leq  0,
\end{equation}
implies the following
\begin{equation*}
Z\subset H(\bar{w},x) \implies \bar{z}\in H(\bar{w},x) \iff x \in \bar{\mathbb{B}}\left(\frac{\bar{w}+\bar{z}}{2},\frac{\|\bar{w}-\bar{z}\|}{2}\right).
\end{equation*}
Therefore, we will limit our attention to  $Q(\bar{w},\cdot,\opert (\cdot))$ given by \eqref{eq:Haugazeau_projection} to be defined on $\setD\subset  \bar{\mathbb{B}}\left(\frac{\bar{w}+\bar{z}}{2},\frac{\|\bar{w}-\bar{z}\|}{2}\right)$. 

Let us note that for $x=\bar{w}$ we have $H(\bar{w},x)=\hilbertH\times \hilbertG$. This motivates us to restrict our investigations  to set $\hatD:=\setD \setminus B(\bar{w},r)$ for some $r>0$ such that $\hatD$ is nonempty.

System \eqref{system:Haugazeau} is 
an autonomous  dynamical system  of the form
\begin{align}\tag{DS}\label{dynamics:DS_2}
\begin{aligned}
& \dot{x}(t)=F(x(t)),\quad t\geq   0,\\
& x(0)=x_0\in \hat{\setD}\setminus \{ \bar{z} \}, 
\end{aligned}	
\end{align}
where $F:\ \hatD \rightarrow \hilbertX$, $\hilbertX$ - Hilbert space, is a continuous function, locally Lipschitz on $\hatD$ except a single point $\bar{z}\in\hatD$, and $\hatD$ is a closed  and  bounded  set in $\hilbertX$. Indeed, when $F(x):=Q(\bar{w},x,\opert x)-x$, where $\opert: \hilbertH\times \hilbertG\rightarrow \hilbertH\times \hilbertG$ is defined as in  \eqref{block} and $Q: (\hilbertH\times \hilbertG)^3 \rightarrow \hilbertH \times \hilbertG$ is defined in  \eqref{Qblock}, the system \eqref{dynamics:DS_2} reduces  to~\eqref{system:Haugazeau}. For other applications we refer the reader to Section \ref{section:PDS}.

%


A survey of existing results on solvability and uniqueness of solutions 
going beyond the classical Cauchy-Picard theorem
from finite to infinite settings journey can be found in~\cite{hajekvivi10}.

Main difficulties  in investigating  the existence  to autonomous ODE in infinite-dimensional settings are due to the lack of compactness, see~\cite[Remark 5.1.1]{MR0460832}. For instance,  the continuity  of the right-hand side vector field $F$ is not enough to obtain the counterpart of  Peano's theorem in infinite-dimensional spaces~\cite{dieudonne50}, even in Hilbert spaces~\cite{yorke70}. 

In \cite{godunov74} Godunov proved that in every infinite-dimensional Banach space there exists a continuous vector field $F$ such that there is no solution to the related \eqref{dynamics:DS_2} whereas the global Lipschitz condition, due to Cauchy-Lipschitz-Picard-Lindeloff, of the right-hand side field ensures the uniqueness and/or extendability of the solution, see~\cite[Theorem 7.3]{brezis2011}.
Some attempts to weaken the global Lipschitz condition of the right-hand side vector field have been done in the context of the existence of solutions, see, e.g., \cite[Theorem 5.1.1]{MR0460832} and ~\cite{hajek2010,li75,szep71,MR2143512} and the references therein. It is observed that the local Lipschitzness of the vector filed allows to prove the local existence and uniqueness for the related  
problems. For instance, one can adapt \cite[Theorem 5.1.1]{MR0460832} to the case of autonomous differential system in the following way
\begin{corollary}\label{theorem:existenceonrectangle}
	Define the rectangle $R_0 = \{ x \in \hilbertX \mid \|x-x_0\|\leq \beta \}$ for some $\beta>0$.
	Let $f: R_0 \rightarrow \hilbertX$. Assume that $\|f(x)\| \leq \tilde{M}$ for $x\in R_0$ and $\|f(x_1)-f(x_2)\|\leq K \|x_1-x_2\|$ for $x_1,x_2\in R_0$, where $K$ and $\tilde{M}$ are nonnegative constants. Let $\alpha>0$  such that $\alpha \leq \frac{\beta}{\tilde{M}}$. Then there exists one and only one (strongly) continuously differentiable function $x(t)$ satisfying
	\begin{align*}
	& \dot{x}(t)=f(x(t)),\quad |t-t_0|\leq \alpha;\quad
	x(t_0)=x_0.
	\end{align*}
\end{corollary}

Let us note that Corollary \ref{theorem:existenceonrectangle} is non-applicable to system \eqref{dynamics:DS_2} in case when $x_0 \notin \operatorname{int} \hatD$ (see also Remark~\ref{remark:fixpointsoutside} below). Moreover, it was shown that local Lipschitzness condition is not enough to guarantee existence of trajectories on $[t_0,+\infty)$ (see e.g.,  \cite{MR2034199} and references therein). Instead of this, in Sections \ref{section:existenceanduniqueness} and \ref{section:extendability} we will be using modified standard techniques to show the existence and uniqueness of solutions to \eqref{dynamics:DS_2}. 


In \cite{MR3443764} a smooth vector field is constructed such that the respective autonomous dynamical system has a bounded maximal solution which is not globally defined. 

In finite-dimensional settings, under the assumption of local Lipschitzness and some boundedness of the vector field, the existence and uniqueness of the trajectory on $[t_0,+\infty)$ are shown in \cite{XiaWang2000} by Xia and Wang. The authors applied their results to investigations of projected dynamical systems.













\section{Preliminaries}\label{section:preliminaries}

In this section we formulate the system \eqref{system:Haugazeau} (and \eqref{dynamics:DS_2}) in the  general  form.

Let  $\bar{w},\bar{z} \in \hilbertX$
and the associated norm in Hilbert space $\hilbertX$ be defined as $\|\cdot \|=\sqrt{\langle \cdot \mid \cdot \rangle}$. Let $\setD\subset \hilbertX$ be a  closed convex subset of $\hilbertX$  such that $\bar{w},\bar{z}\in \setD$ and 
\begin{equation}\label{inequality:setD}
\langle  \bar{z} - x \mid \bar{w}-x \rangle \leq  0\quad  \text{for all}\ x\in \setD.
\end{equation}
Note that the condition \eqref{inequality:setD} immediately implies that $\bar{w}$ and $\bar{z}$ are boundary points of the set $\setD$.


Let $r$ be such that $\|\bar{w}-\bar{z}\|^2>r>0$. Throughout this paper, we consider set $\hat{\setD}$ related to $\setD$ (see Figure \ref{figure1}):
\begin{align}\label{set:Dhat} 
\hat{\setD}=\{ x\in \setD \mid \|x-\bar{w}\|^2\geq r  \}. 
\end{align}

We consider the following Cauchy problem
\begin{align}\label{dynamics:DS}\tag{DS-0}
\begin{aligned}
& \dot{x}(t)=F(x(t)),\quad \text{ }t\geq t_0 \geq 0,\\
& x(t_0)=x_{00}\in \hat{\setD}\setminus \{ \bar{z} \}, 
\end{aligned}	
\end{align}
where   $F:\ \hat{\setD} \rightarrow \hilbertX$ is a  continuous function on $\hat{\setD}$ and locally Lipschitz on $\hat{\setD}\setminus \{ \bar{z} \}$ and bounded on $\hat{\setD}$ ($\|F(x)\|\leq M$, $M>0$, $x\in \hat{\setD}$). 

Moreover, we assume:
\begin{enumerate}[label=(\Alph*)]
	\item\label{assumption:A1}  $\bar{z}$ is the only zero point of $F$ in $\hat{\setD}$, i.e. $F(x)=0$ iff $x=\bar{z}$. 
	\item\label{assumption:A2} for all $x\in \hat{\setD}$, for all $h\in [0,1]$ we have $x+hF(x)\in \hat{\setD}$
\end{enumerate}
Together with assumptions \ref{assumption:A1}, \ref{assumption:A2} we also consider the following assumption related to the behaviour of projection\footnote{Here, for  $f(x):=F(x)+x$ (so that $F(x)=f(x)-x$) we have that $\bar{z}\in H(\bar{w},f(x))$.}:
\begin{enumerate}[label=(\Alph*)] \setcounter{enumi}{2}
	\item \label{assumption:A3} $\langle F(x) \mid \bar{w} - x \rangle \leq 0$ for all $x\in \hat{\setD}$.
\end{enumerate}

\begin{figure}[H]
	\begin{center}
		\begin{subfigure}{0.3\linewidth}
			\includegraphics[trim={300px 100px 300px 100px},clip,scale=0.3]{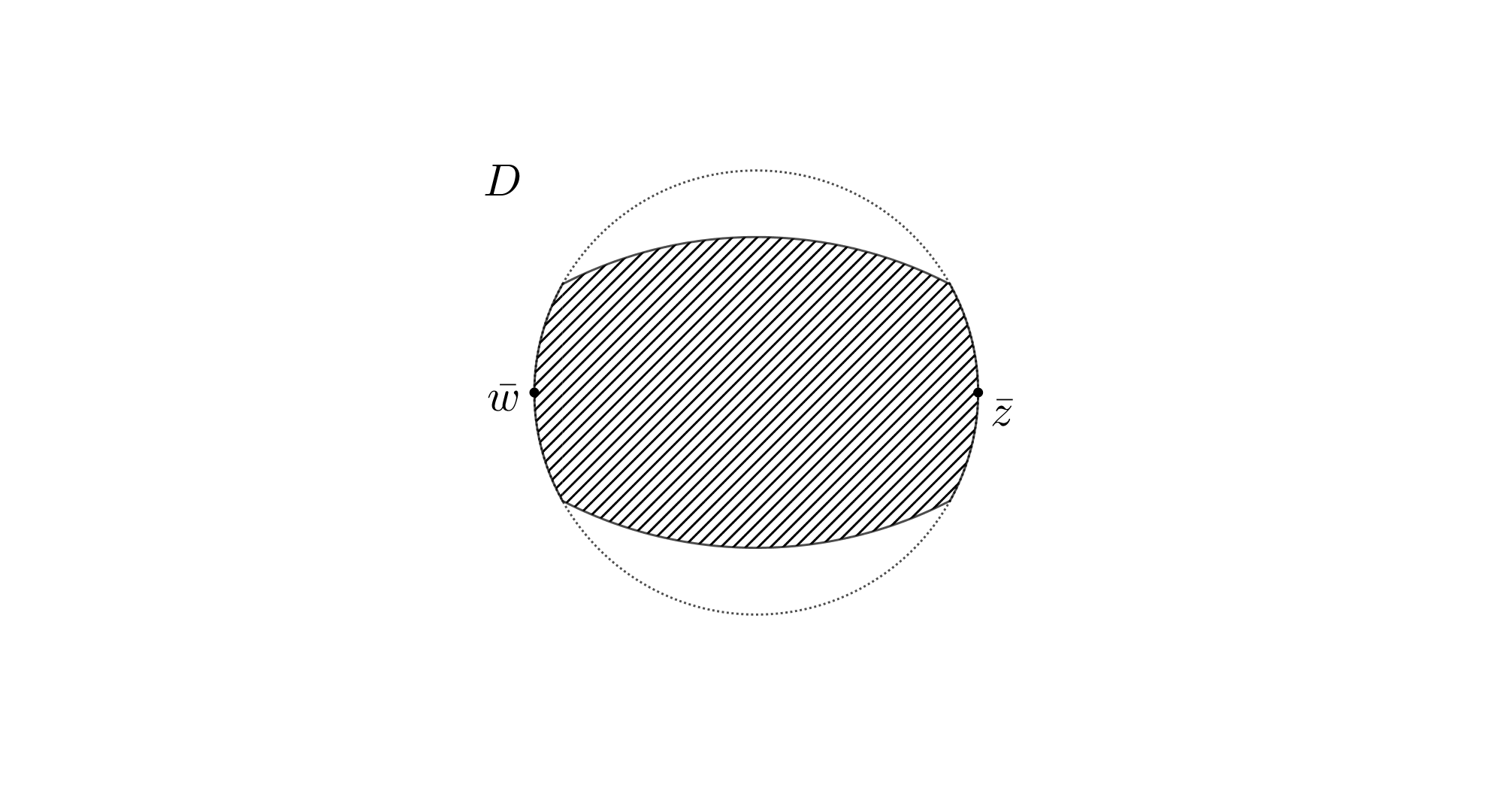}
		\end{subfigure}\hspace{0.05\linewidth}
		\begin{subfigure}{0.3\linewidth}
			\includegraphics[trim={300px 100px 300px 100px},clip,scale=0.3]{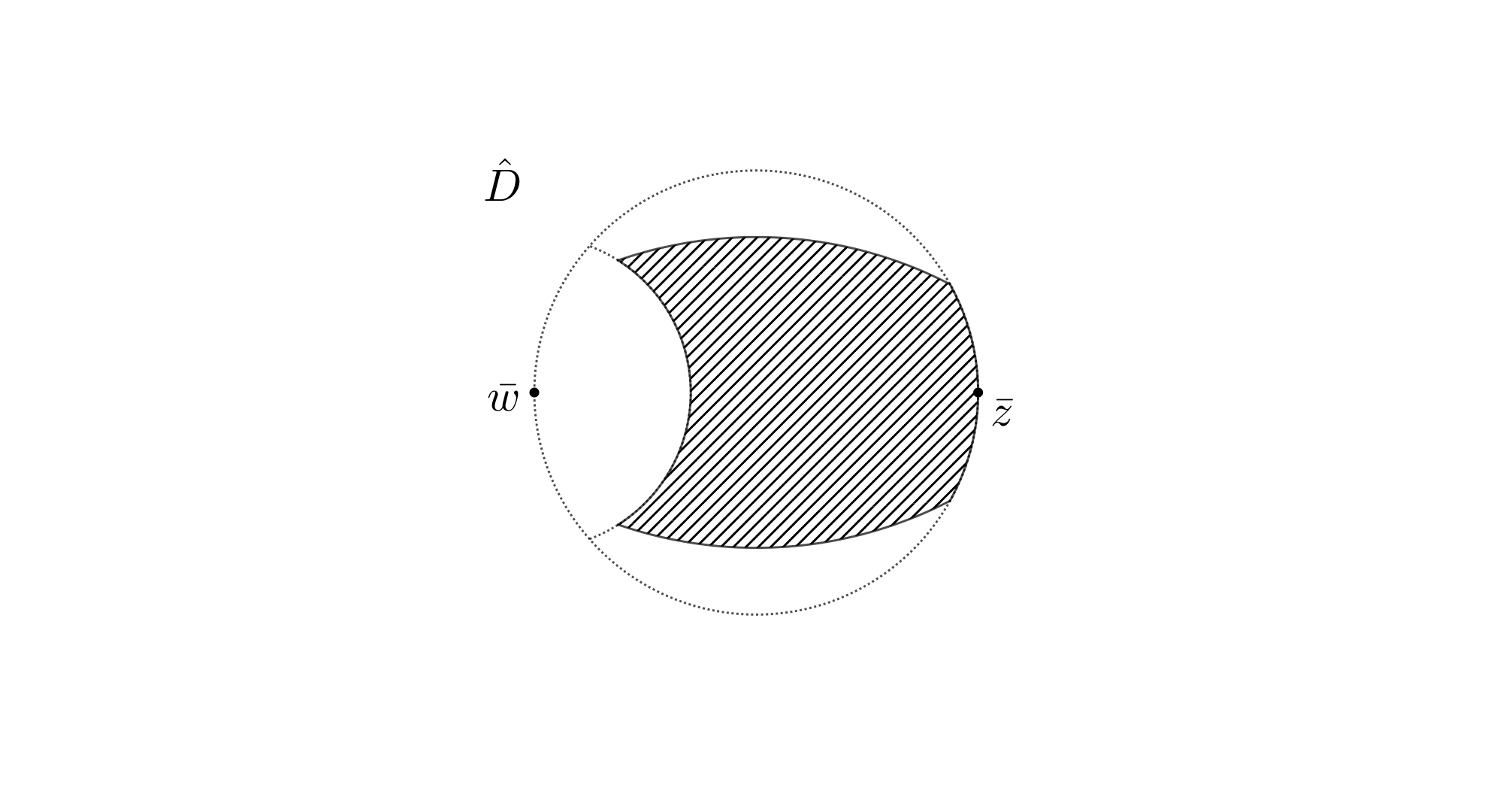}
		\end{subfigure}
		\caption{Illustration of the considered sets.}
		\label{figure1}
	\end{center}
\end{figure}
\begin{remark}
	The motivation for considering a nonconvex set $\hat{\setD}$ comes from the following observation. Consider $F:\ \setD \rightarrow \hilbertX$ defined as \begin{equation}\label{forumla:projection}
	F(x)=P_{\multifC(x)}(\bar{w}),
	\end{equation} where
	$P_{\multifC(x)}(\bar{w})$ is the projection of $\bar{w}$ onto $\multifC(x)$, 
	$\multifC :\ \setD \rightrightarrows \hilbertX$ is a multifunction  given by $\multifC(x)=H(\bar{w},x)\cap H(x,g(x))$ (see formula \eqref{Qblock} for $H(\cdot,\cdot)$)
	and $g:\ \hilbertX \rightarrow \hilbertX $  satisfies $\bar{z}\in H(x,g(x))$ for all $x\in \hilbertX$. 
	Under a suitable assumption on $g$, the function $F$ given by \eqref{forumla:projection} is locally Lipschitz on $\setD\setminus \{\bar{w},\bar{z}\}$ (see e.g. \cite{MR4178884}), continuous on $\setD\setminus \{\bar{w}\}$ and bounded on $\setD$. 
\end{remark}

Throughout the  paper we use the following concept of solutions for dynamical systems \eqref{dynamics:DS} and  \eqref{dynamics:DS_2} and its extendibility.
\begin{definition} 
	\label{definition1}
	Let 
	$$
	{\mathcal T}=[t_{0};T), \ t_{0}<T\leq+\infty\ \  \text{or } {\mathcal T}=[t_{0};T],\ t_{0}<T<+\infty.
	$$
	\textit{Solution} of \begin{align}\label{dynamics:DS_A}\tag{DS-A}
	\begin{aligned}
	& \dot{x}(t)=F(x(t)),\quad \text{ }t\geq t_0 \geq 0,\\
	& x(t_0)=x_{00}\in A\setminus \{ \bar{z} \}, 
	\end{aligned}	
	\end{align}
	where $F:\ A \rightarrow \hilbertX$, $A\subseteq \hilbertX$, 
	on interval ${\mathcal T}$
	is any 
	function 
	$$
	x(\cdot)\in C^{1}({\mathcal T},A)
	$$
	satisfying
	\begin{enumerate} 
		\item initial condition $x(t_{0})=x_0$;
		\item equation $\dot{x}(t)=F(x(t))$ for all $t\in {\mathcal T}$, where the differentiation is understood in the sense of strong derivative on space $\hilbertX$
		, where at the boundary point of the interval ${\mathcal T}$, in the case when it belongs to ${\mathcal T}$, the differentiation is understood in the one-sided way.
	\end{enumerate}
\end{definition}

\begin{definition}
	A solution $x(t)$ to 
	problem \eqref{dynamics:DS_2}. on interval $\mathcal{T}_1=[0,T]$ (or $\mathcal{T}_1=[0,T)$) is called non-extendable if there is no solution $x_2(\cdot)\in C^1(\mathcal{T}_2,\hat{\setD})$ on any interval $\mathcal{T}_2$ of this problem satisfying conditions:
	\begin{enumerate}
		\item $\mathcal{T}_2\supsetneq \mathcal{T}_1$;
		\item $\forall t\in \mathcal{T}_1,\quad x_2(t)=x(t)$.
	\end{enumerate}
\end{definition}
\begin{remark}\label{remark:restriction}
	If $x(t)$ is a solution of Cauchy problem \eqref{dynamics:DS} on interval $\mathcal{T}=[0,T]$ (or $\mathcal{T}=[0,T)$), then restriction of $x(t)$ on any interval $\mathcal{T}_1=[t_0,t_1]\subset\mathcal{T}$ (or $\mathcal{T}_1=[t_0,t_1)\subset \mathcal{T}$) is a solution of Cauchy problem \eqref{dynamics:DS} on ${\mathcal T}_1$ with initial condition $x_0=x(t_0)$.
\end{remark}
The main results on the existence, uniqueness and extendibility of solutions to \eqref{dynamics:DS_2} read as follows.

\begin{theorem}[Existence and uniqueness]
	\label{proposition_prolongation}
	Suppose that assumptions \ref{assumption:A1}, \ref{assumption:A2} and \ref{assumption:A3} hold. There exists a unique solution of \eqref{dynamics:DS}   on $[t_0,+\infty)$. 
\end{theorem}

\begin{theorem}[Behavior at $+\infty$]\label{proposition:conditions_convergence}
	Let $x(t)$ be a solution of~\eqref{dynamics:DS} on $[t_0,+\infty)$. Assume that 
	for every increasing sequence $\{t_n\}_{n\in\mathbb{N}}$, $t_n\rightarrow +\infty$
	\begin{equation}\label{assumption:weak2}
	x(t_{n})\rightharpoonup \tilde{x} \implies \tilde{x}=\bar{z},
	\end{equation}
	where $x(t)$ is a unique solution of \eqref{dynamics:DS}.
	Then the trajectory $x(t)$ satisfies the condition $\lim_{t\rightarrow +\infty} x(t)=\bar{z}$, where convergence is understood in the sense of the norm of $\hilbertX$.
\end{theorem}
\begin{remark}
	Condition \eqref{assumption:weak2} can be seen as a continuous 
	analogue of condition (iv) of Proposition 2.1 of \cite{Alotaibi_2015_best}. Namely, to obtain the strong convergence of the sequence generated by \eqref{scheme:haugazeau} it is assumed in Proposition 2.1 of \cite{Alotaibi_2015_best} that for any strictly increasing sequence $\{k_n\}\subset \mathbb{N}$ the following implication holds:
	\begin{equation*}
	x_{k_n}\rightharpoonup \tilde{x} \implies \tilde{x}=\bar{z}.
	\end{equation*}
\end{remark}


\section{
	Solutions to \texorpdfstring{\eqref{dynamics:DS}}{dynamicsDS2} on closed intervals}\label{section:existenceanduniqueness}

In this section we consider the existence and uniqueness of solutions to \eqref{dynamics:DS} defined on closed intervals, namely, $[t_0,T]$, where $T>t_0$ is finite. In deriving  existence and uniqueness results, we modify  two standard approaches (with the help of assumptions~\ref{assumption:A1}-\ref{assumption:A3}):  Euler method (Section \ref{euler}) and  contraction mapping principle (Section \ref{contraction}). To this aim we will use the following proposition.

\begin{proposition}\label{proposition:increasing}
	Assume that \ref{assumption:A3} holds. Then any 
	solution $x(t)$ of \eqref{dynamics:DS} satisfies the condition
	\begin{equation*}
	\|x(t)-\bar{w}\|\quad  \text{is nondecreasing with respect to } t\geq t_0.
	\end{equation*}
\end{proposition}
\begin{proof}
	Let us note that 
	$x(t)$ is continuously differentiable  on   $ [t_0,+\infty) $, therefore by \ref{assumption:A3} we have
	\begin{align*}
	\frac{1}{2} \frac{d}{dt} \|x(t)-\bar{w}\|^2&= \langle \dot{x}(t) \mid x(t)-\bar{w}\rangle\\
	&= \langle F(x(t)) \mid x(t)-\bar{w}\rangle \geq 0.
	\end{align*}   
\end{proof}

Now we show the uniqueness of trajectories.
\begin{proposition}\label{proposition:uniqueness}
	Let $t_0\geq 0$ and let $x_0\in \hat{\setD}\setminus \{ \bar{z} \}$. Assume that assumptions \ref{assumption:A1} and \ref{assumption:A3} holds. If  \eqref{dynamics:DS} is solvable in a given interval $[t_0,T]$, then the solution is unique on this interval.
\end{proposition}
\begin{proof}
	Now we show the uniqueness of solutions of \eqref{dynamics:DS} on $[t_0,T]$. 
	Suppose that $x_1(\cdot)$ and $x_2(\cdot)$ solve \eqref{dynamics:DS} on interval $[t_0,T]$.
	Let $\bar{t}\in [t_0,T]$ be such that 
	\begin{equation}\label{assumtion:infimum}
	\bar{t}:=\sup \{t\in [t_0,T] \mid \|x_1(t)- x_2(t)\|= 0\}.
	\end{equation}
	Let us note that $x_1(t_0)=x_{00}=x_2(t_0)$. Consider two cases:
	\begin{enumerate}
		\item[Case 1]: $x_1(\bar{t})=x_2(\bar{t})=\bar{z}$. Then, by Proposition \ref{proposition:increasing}, 
		$\|x_1(t)-\bar{w}\|\geq \|\bar{z}-\bar{w}\|$ and  $\|x_2(t)-\bar{w}\|\geq \|\bar{z}-\bar{w}\|$ for $t\geq \bar{t}$. However, since $\bar{z}\in \hatD$ and, by \eqref{fact:subset_ball}, $\hatD\subset \setD \subset \bar{\mathbb{B}}(\frac{\bar{w}+\bar{z}}{2},\frac{\|\bar{w}-\bar{z}\|}{2})$, 
		\begin{equation*}
		\{ x\in \hilbertX \mid  \|x-\bar{w}\| \geq \|\bar{w}-\bar{z}\| \} \cap \hat{\setD}= \{\bar{z}\}.
		\end{equation*}
		Therefore, by assumption (A), $x_1(t)=x_2(t)=\bar{z}$ for all $t\in [\bar{t},T]$.
		\item[Case 2]: $x_1(\bar{t})=x_2(\bar{t})\neq \bar{z}$ and $\bar{t}< T$. Then, by the local Lipschitzness of $F(\cdot)$ on $\hat{\setD}\setminus \{\bar{z}\}$ there exists a neighbourhood of $x_1(\bar{t})$, namely $U(x_1(\bar{t}))$ such that $F$ is locally Lipschitz in $U(x_1(\bar{t}))$ with some constant $L_{x_1(\bar{t})}$, i.e.,
		\begin{equation*}
		\forall x^1,x^2 \in U(x_1(\bar{t})) \quad \|F(x^1)-F(x^2)\|\leq L_{x_1(\bar{t})} \|x^1-x^2\|.
		\end{equation*}
		Since $x_1$ and $x_2$ are Lipschitz functions with constant $M$ there exists a neighbourhood $V(\bar{t})\cap [t_0,T]$ such that
		\begin{equation*}
		\forall t\in V(\bar{t})\cap [t_0,T] \quad x_1(t)\in U(x_1(\bar{t})) \quad \wedge \quad x_2(t)\in U(x_1(\bar{t})).
		\end{equation*}
		Then for $t\in V(\bar{t})\cap [t_0,T]$
		\begin{align*}
		&\frac{d}{dt}\left(\frac{1}{2}\|x_1(t)-x_2(t)\|^2\right)=\langle \dot{x}_1(t)-\dot{x}_2(t) \mid x_1(t)-x_2(t)\rangle \\
		&= \langle F(x_1(t))-F(x_2(t)) \mid x_1(t)-x_2(t)\rangle\leq L_{x_1(\bar{t})}\|x_1(t)-x_2(t)\|^2.
		\end{align*}
		By using Gronwall's inequality for the function $t\rightarrow\|x_1(t)-x_2(t)\|^2$ we obtain that $\|x_1(t)-x_2(t)\|^2\leq 0$, i.e., $x_1(t)=x_2(t)$ for $t\in V(\bar{t})\cap [t_0,T]$. This contradicts \ref{assumtion:infimum} with $\bar{t}\neq T$.
	\end{enumerate}  
	
\end{proof}

\begin{proposition}\label{proposition:voltera_solution}
	$x(t)$ is a solution of \eqref{dynamics:DS_A} on $\mathcal{I}=[t_0,T]$ ($T>t_0$ is arbitrary) if and only if it satisfies the condition
	\begin{equation}\label{solutions:integral_condition}
	x(t)=x_0+\int_{t_0}^{t} F(x(s))\, ds,\quad \forall t\in \mathcal{I},
	\end{equation}	
	where the integral is understood in the sense of Riemann and $x(t)\in \hat{\setD}$, $t\in \mathcal{I}$.
\end{proposition}

Let us define
\begin{equation*}
\mathbb{B}_{t_0,T}:= C([t_0,t_0+T], \hat{D}  )
\end{equation*}
and
\begin{equation*}
\mathbb{B}_{t_0,x_0,T}^R:= \{ x\in \mathbb{B}_{t_0,T} \mid \sup_{t\in [t_0,t_0+T]} \|x(t)-x_0\|\leq R     \}.
\end{equation*}

Let us note that $\mathbb{B}_{t_0,T}$ is a complete metric space due to the fact that $\hat{\setD}$ is a closed subset of a Hilbert space $\hilbertX$.
Moreover, in the sequel we consider on $\hat{\setD}$   the topology induced by the topology of the space.

\subsection{Euler method}
\label{euler}
We start with the following construction of Euler trajectories. 

For any $\lambda\in (0,1]$ define $c_n^{\lambda}$, $n=0,1,\dots$
as follows
\begin{equation}\label{sequence:nodes2}
c_0^{\lambda}:=x_0,\ \  c_{n+1}^{\lambda}:=c_{n}^{\lambda}+\lambda F(c_{n}^{\lambda}), \ n=0,1,\cdots.
\end{equation}
Then, for any $\lambda \in (0,1]$ define a continuous trajectory on $[t_0,T]$ as follows
\begin{align}\label{formula:fun_seq_2_2}
\begin{aligned}
c_\lambda(t)=c_n^\lambda+(t-t_0-n\lambda)F(c_n^\lambda),\quad & t\in [t_0+n\lambda,t_0+(n+1)\lambda]
\\ & n=0,1,\cdots.
\end{aligned}
\end{align}

\begin{proposition}\label{proposition:existence_finite2}
	Let $t_0>0$ and let $x_0\in \hat{\setD}\setminus \{ \bar{z} \}$. Assume that 
	\ref{assumption:A2} 
	hold.
	\begin{enumerate}
		\item If $\hilbertX$ is finite-dimensional, then for all $T>t_0$ there exists
		a 
		solution of $x(t)$ of \eqref{dynamics:DS} on $[t_0,T]$ in the class $\mathbb{B}_{t_0,T}$,
		\item If $\hilbertX$ is infinite-dimensional, then there exists $R>0$ and  $T>t_0$ such that there exists
		a 
		solution of $x(t)$ of \eqref{dynamics:DS} on $[t_0,T]$ in the class $\mathbb{B}_{t_0,x_0,T}^R$,
	\end{enumerate}
\end{proposition}	
\begin{proof}Let us start with the initial settings.
	\begin{enumerate}
		\item\label{case:first} In case $\hilbertX$ is finite-dimensional we take any $T>t_0$. Let us note that in this case $\hat{\setD}$ is closed and  bounded, hence compact. Since $F$ is continuous on $\hat{\setD}$, $F$ is uniformly continuous, i.e. 
		\begin{equation*}
		\forall \varepsilon>0\  \exists\delta>0\ \forall x_1,x_2\in \hat{\setD} \quad  \|x_1-x_2\|< \delta \implies \|F(x_1)-F(x_2)\|< \varepsilon.
		\end{equation*}
		\item\label{case:second} In case  $\hilbertX$ is infinite-dimensional let $T=\frac{R}{M}+t_0$, where $R$ is  such that $F(\cdot)$ is Lipschitz on $B(x_0,R)$. Let $m_\lambda:=\lceil(T-t_0)\lambda^{-1}\rceil$. 	Let us note that, by the fact that $x_0\in \hat{\setD}\setminus \{\bar{z}\}$ and, by assumption \ref{assumption:A2}, for any $\lambda\in (0,1]$ and  all $t\in [t_0,T]$, $c_\lambda(t)\in \hat{\setD}$.
		For any $\lambda \in (0,1]$ function $c_\lambda(\cdot)$ given by \eqref{formula:fun_seq_2_2} is differentiable on $[t_0,T]\setminus \{t_0,t_0+\lambda,\dots, t_0+m_\lambda \lambda \}$ as a piecewise affine function.
		
		For all $\lambda\in (0,1]$ and any  $t\in [t_0,T]$ ($t=t_0+a\lambda+\tilde{t}$, $a\in \mathbb{N}$, $0\leq \tilde{t}<\lambda$) we have
		\begin{align*}
		&\|c_\lambda(t)-x_0\|= \| x_0 + \lambda \sum_{n=0}^{a-1} F(c_{\lambda}(t_0+n\lambda))+\tilde{t}F(c_{\lambda}(t_0+a\lambda)) - x_0\|\\
		&\leq \lambda \sum_{i=0}^{a-1} M + \tilde{t} M= M(a\lambda+\tilde{t})\leq M (T-t_0) =\frac{R}{M} M = R.
		\end{align*}
		Let us note that in this case $F$ is uniformly continuous on $B(x_0,R)\cap \hat{\setD}$, i.e. 
		\begin{equation*}
		\forall \varepsilon>0\  \exists\delta>0\ \forall x_1,x_2\in B(x_0,R)\cap \hat{\setD} \quad  \|x_1-x_2\|< \delta \implies \|F(x_1)-F(x_2)\|< \varepsilon.
		\end{equation*}
	\end{enumerate}
	Now let us continue the proof in both cases \ref{case:first}. and \ref{case:second}. together.	For any $\lambda \in (0,1]$ define
	\begin{equation}
	\Delta_\lambda(t):=\left\{ \begin{array}{ll}
	\dot{c}_\lambda(t) - F(c_\lambda(t)),\quad \mbox{} & t_0+n\lambda<t<\min\{t_0+(n+1)\lambda,T\},\\
	& n=0,1\dots,m_\lambda, \\
	0  & t= t_0,t_0+ \lambda,\dots,t_0+m_\lambda\lambda.
	\end{array} \right.
	\end{equation}
	Note that for all $t\in [t_0,T]$,
	\begin{equation*}
	c_\lambda(t)=x_0+\int_{t_0}^{t} \dot{c}_\lambda(s)\, ds = x_0+\int_{t_0}^{t} F(c_\lambda(s)) + \Delta_\lambda(s)  \, ds,
	\end{equation*}
	We have
	\begin{align*}
	\|\Delta_\lambda(t)\|=\| F(c_n^\lambda)-F(c_\lambda(t))\|,\quad & t\in [t_0+n\lambda ,\min\{t_0+(n+1)\lambda,T\}],\\
	&n=0,1,\dots,m_\lambda. 
	\end{align*}
	Let us note that $c_\lambda(\cdot)$ is Lipschitz continuous on $[t_0,T]$ because it is differentiable almost everywhere and the norm of its derivative is bounded by $M$. Therefore 
	\begin{align*}
	&\forall n=0,1\dots,m_\lambda\\
	&\sup_{t\in [t_0+n\lambda,\min\{ t_0+(n+1)\lambda,T\}]}  \|c_\lambda(t_0+
	n\lambda )-c_k(t)\|\leq M |n\lambda-t|\leq M\lambda
	\end{align*}

	Fix any $\varepsilon>0$ and take $\lambda \in (0,1]$ such that $ M \lambda <\delta$. Then for all  $n=0,\dots,m_\lambda$ we have 
	\begin{align*}
	&\sup_{t\in [t_0+n\lambda,\min\{ t_0+(n+1)\lambda,T\}]} \|c_\lambda(t_0+n\lambda)-c_\lambda(t)\|\\
	&=\sup_{t\in [t_0+n\lambda,\min\{ t_0+(n+1)\lambda,T\}]} \| (t-(t_0+n\lambda))F(c_n^\lambda) \|\leq M\lambda< \delta,
	\end{align*} and consequently
	\begin{align*}
	&\forall n=0,\dots,m_\lambda\\ 
	& \sup_{t\in [t_0+n\lambda,\min\{ t_0+(n+1)\lambda,T\}]} \| F(c_n^\lambda)-F(c_\lambda(t)\|<\varepsilon.
	\end{align*}
	Hence, for all $\lambda < \frac{\delta}{M}$, we have
	\begin{align*}
	\forall n=0,\dots,m_\lambda \ \forall t\in  [t_0+n\lambda,\min\{ t_0+(n+1)\lambda,T\}]\quad  \|\Delta_\lambda(t)\|<\varepsilon.
	\end{align*}
	Thus,
	\begin{equation*}
	\|\Delta_\lambda(\cdot)\|_{+\infty}\rightarrow 0 \text{ as } \lambda\rightarrow 0 \text{ on } [t_0,T].
	\end{equation*}

	Let $\{\lambda_k\}_{k\in\mathbb{N}}$ be a sequence in $(0,1]$ such that $\lambda_k\rightarrow 0$ as $k\rightarrow 0$. 
	By the Ascoli-Arzela Theorem, there exists a uniformly convergent subsequence of $\{ c_{\lambda_k}(t)  \}_{k\in \mathbb{N}}$, namely $\{ c_{\lambda_{k_i}}(t)  \}_{i\in \mathbb{N}}$, which converges to $x(t)=\lim_{i\rightarrow +\infty}  c_{\lambda_{k_i}}(t)$ for $t\in [t_0,T]$, i.e.
	\begin{align}\label{convergence:trajectories_to_solution}
	\begin{aligned}
	&\exists \{\lambda_{k_i}\}_{i\in \mathbb{N}}\ \forall \varepsilon>0\ \exists i_0\in \mathbb{N}\ \forall i\geq i_0 \ \forall t\in [t_0,T]    \\
	&\|x(t)-c_{\lambda_{k_i}}(t)\|<\varepsilon.
	\end{aligned}
	\end{align}
	Therefore, for all $t\in [t_0,T]$,
	\begin{align*}
	c_{\lambda_{k_i}}(t)=x_0+\int_{t_0}^{t} F(c_{\lambda_{k_i}}(s)) + \Delta_{\lambda_{k_i}}(s)\, ds,\ \ \ 
	x(t)=x_0+\int_{t_0}^{t} F(x(s))\, ds.
	\end{align*}
	
	By Proposition \ref{proposition:voltera_solution},  $x(t)$ is a solution of \eqref{dynamics:DS} on $[t_0,T]$. Since $c_{\lambda_{k_i}}(t)\in \hat{D}$, $i\in \mathbb{N}$, $t\in [t_0,T]$, by the closedness of $\hat{\setD}$, we obtain that $x(t)\in \hat{D}$, $[t_0,T]$.   
\end{proof}

\begin{corollary}\label{lemma1}
	Let $t_0\geq 0$, $x_0\in \hat{\setD}\setminus \{ \bar{z}\}$ be arbitrary fixed. Assume that assumptions \ref{assumption:A1}, \ref{assumption:A2}, \ref{assumption:A3} are satisfied. Then there exist  $R>0$, $T^\prime>0$ such that for all $T\in [t_0,T^\prime)$ there exists solution to \eqref{dynamics:DS} on $[t_0,t_0+T]$ and it is unique in the class $\mathbb{B}_{t_0,x_0,T}^R$.
\end{corollary}

\begin{proof}
	The proof follows from Proposition \ref{proposition:uniqueness} and Proposition \ref{proposition:existence_finite2}.  
\end{proof}

\subsection{Contraction mapping principle for an extended vector field $F$.}\label{contraction}
We consider the following Cauchy problem
\begin{align}\tag{DS-1}
\label{dynamics:DS_3}
\begin{aligned}
& \dot{x}(t)=\tilde{F}(x(t)),\quad \text{ }t\geq 0,\\
& x(0)=x_{00}\in \hat{\setD}\setminus \{ \bar{z} \}, 
\end{aligned}	
\end{align}
where $\tilde{F}:\ \hilbertX\rightarrow \hilbertX$ is such that $\tilde{F}(x)=F(x)$ for all $x\in \hat{D}$ and $\tilde{F}$ is continuous on $\hilbertX$.

\begin{lemma}(\cite[Lemma 1.2]{ODE_in_Banach_Deimling})\label{lem:ext}
	Let $X,Y$ be Banach spaces, $\Omega\subset X$ closed and $f:\Omega\rightarrow Y$ continuous. Then there is a continuous extension $\tilde{f}:X\to Y$ of $f$ such that $\tilde{f}(X)\subset {\rm conv}f(\Omega)$ (:=convex hull of $f(\Omega)$).
\end{lemma}



\begin{proposition}
	\label{contraction2}
	Let $t_0\ge 0$, $x_0\in \hat{\setD}\setminus \{ \bar{z} \}$. 
	Then there exists $T>0$ such that there exists a solution of \eqref{dynamics:DS_3} on interval $[t_0,t_0+T]$.
\end{proposition}

\begin{proof}
	For a given $x\in C([t_0,t_0+T];\hilbertX)$, 
	define $S[x]$ to be the function on $[t_0,t_0+T]$, given by
	\begin{align}\label{eq:def:sol_2}
	S[x](t):= x_0+\int_{t_0}^{t}\tilde{F}(x(\tau))\,d\tau,\quad t\in [t_0,t_0+T],
	\end{align}
	where $\tilde{F}$ is an extension of $F$ given by Lemma~\ref{lem:ext}. In the following, the boundedness of $F$ or $\tilde{F}$ will be used as per their restrictive sense.
	\begin{itemize}
		\item[Step 1.] If $x\in C([t_0,t_0+T];\hat{D})$, 
		then $S(x)$  makes sense, since the right hand side is well defined.
		\item[Step 2.] Let us prove that $S[x](\cdot)\in C([t_0,t_0+T];\hilbertX)$ for any $T>0$ and for $x\in C([t_0,t_0+T];\hilbertX)$.
		Assume $t_1,t_2\in [t_0,t_0+T]$ with $t_1<t_2$. It is evident that
		\begin{align}\label{eq:def:sol:cont_2}
		S[x](t_2)= S[x](t_1)+\int_{t_1}^{t_2}\tilde{F}(x(\tau))\,d\tau.
		\end{align}
		Then the continuity of $S[x]$ gives us as $t_2\to t_1$,
		\begin{align*}
		\|S[x](t_2)-S[x](t_1)\|_{\hilbertX}=\left\|\int_{t_1}^{t_2}\tilde{F}(x(\tau))\,d\tau\right\|_{\hilbertX}\le \max_{\tau\in[t_0,t_0+T]}\|\tilde{F}(x(\tau))\|_{\hilbertX}\cdot|t_2-t_1|.
		\end{align*}
		Thus $S:C([t_0,t_0+T];\hilbertX)\longrightarrow C([t_0,t_0+T];\hilbertX)$. 
		\item[Step 3.] 
		
		Denote $C_0:=C([t_0,t_0+T];\hilbertX)$. Consider the following form of a ball in $C_0$, where we intend to look for a fixed point.
		\begin{align*}
		C_{0D}:=\left\{x(t)\in  C_0:\quad|x-x_0|_{C_{0}}\equiv \max_{t\in[t_0,t_0+T]}\|x(t)-x_0 \|_\hilbertX\le 1/2,\ x_0\in \hat{D} \right\}.
		\end{align*}
		Clearly, $C_{0D}(\subseteq C_0)$ is a complete metric space with the metric induced by the norm of $C_0$.
		Let us show that for choosing $T$ small enough the operator $S$ maps $C_{0D}$ into itself and has a fixed point.

		We have, by Step 2, $S[x](\cdot)\in \ C_0$, whenever $x(\cdot)\in C_{0D}$. We now show that $S[x](\cdot)\in C_{0D}$. It follows from~\eqref{eq:def:sol:cont_2} that
		\begin{align*}
		|S[x]-x_0|_{C_{0}}=\max_{t\in[t_0,t_0+T]}\|S[x](t)-x_0\|_{\hilbertX}=\max_{t\in[t_0,t_0+T]}\left\|\int_{0}^{t}\left(\tilde{F}(x(\tau))\right)\,d\tau\right\|_\hilbertX\\
		\le \max_{\tau\in[t_0,t_0+T]}\|\tilde{F}(x(\tau))\|_{\hilbertX} T=: cT,
		\end{align*}
		
		Therefore, for a choice of $T\le 1/2c$,
		\[ |S[x]-x_0|_{C_{0}}\le 1/2.\]
		Hence, $S[x](\cdot)\in C_{0D}$ implies $S:C_{0D}\rightarrow C_{0D}$ for every $T\le 1/2c$.
		\item[Step 4] We shall show now that a sequence $\{ x_n(\cdot)\}_{n\ge 1}\subseteq C_{0D}$ is a Cauchy sequence. Lets start with the initial point $\{x_0\}\in \hat{D}$ be given and define $x_0(\cdot):=x_0$.  Denote $x_1(\cdot):=S[x_0](\cdot)$, and that $x_{n+1}(\cdot):=S[x_n](\cdot),\ n=1,2,\dots$. 
		
		Moreover, the followings hold successively.
		\begin{align*}
		&|x_{n+1}-x_{n}|_{C_{0}}=|S[x_n]-S[x_{n-1}]|_{C_{0}}\le cT|x_n - x_{n-1}|_{C_{0}}\\
		&\le \cdots\le (cT)^{n}|x_1-x_0|_{C_{0}}=(cT)^{n}\max_{t\in[0,T]}\|x_1(t)-x_0\|_{\hilbertX}\le c^n T^{n+1}\|F(x_{0})\|_{\hilbertX}
		\end{align*}
		Let $m, n\in\mathbb{N}$ such that $m > n$ and $cT=\delta\in[0,1]$
		then
		\begin{align*}
		&|x_m - x_n|_{C_{0}}\le |x_m - x_{m-1}|_{C_{0}}+|x_{m-1} - x_{m-2}|_{C_{0}}+\cdots+|x_{n+1}-x_n|_{C_{0}}\\
		&\le  (\delta^{m}+\delta^{m-1}+\cdots+\delta^{n+1})\frac{\|F(x_{0})\|_{\hilbertX}}{c}
		= \delta^{n+1}\frac{\|F(x_{0})\|_{\hilbertX}}{c}\sum_{k=0}^{m-n}\delta^k\\
		&\le \delta^{n+1}\frac{\|F(x_{0})\|_{\hilbertX}}{c}\sum_{k=0}^{\infty}\delta^k
		\stackrel{\delta<1}{=} \frac{\delta^{n+1}\|F(x_{0})\|_{\hilbertX}}{c(1-\delta)}.
		\end{align*}
		Let $\varepsilon>0$. Moreover, since $\delta\in [0, 1)$, we can find a large number $N \in\mathbb{N}$ so that
		\[
		\delta^{N+1}< \varepsilon c(1-\delta)/\|F(x_{0})\|_{\hilbertX}.
		\]
		Therefore, for $m,n>N\in\mathbb{N}$, 
		\[
		|x_m - x_n|_{C_{0}}\le \varepsilon.
		\]
		Hence, we have that the sequence $\{x_n(\cdot)\}_{n\ge 1}\subseteq C_{0D}$ is Cauchy. Therefore, $\{x_n(\cdot)\}_{n\ge 1}$ converges to some $\bar{x}(\cdot)\subseteq C_{0D}$, where  $\bar{x}(\cdot)$ satisfies
		\begin{align}
		\bar{x}(t)= x_0+\int_{0}^{t}\tilde{F}(\bar{x}(\tau))\,d\tau,\quad \forall t\in [t_0,t_0+T].
		\end{align}
		By Proposition \ref{proposition:voltera_solution}, $\bar{x}(\cdot)$ is a solution of \eqref{dynamics:DS_3} for $t\in [t_0,t_0+T]$.
		
	\end{itemize}  
\end{proof}

\begin{remark}\label{remark:fixpointsoutside}
	The proof of the above proposition will not work in the formulation of $\tilde{F}$ defined only on set $\hat{\setD}$. This comes from the fact that the operator
	\begin{equation*}
	S[x](\cdot):\ \hat{\setD} \rightarrow C([t_0,t_0+T],\hilbertX)
	\end{equation*}
	may map a function $x(\cdot)$ outside of $\hat{\setD}$ for which we cannot apply  {\rm Step 4.} in the proof. However, in the case when $x_0\in\text{int}\, \hat{\setD}$,  the following corollary holds.
\end{remark}
\begin{corollary}
	We have the following relationships between \eqref{dynamics:DS_2} and \eqref{dynamics:DS_3}:
	\begin{enumerate}
		\item if $x_0\in \text{int}\, \hat{\setD}$, then there exists a function $x(\cdot)\in C^1([t_0,t_0+T],\hat{\setD})$, which is a unique solution of \eqref{dynamics:DS_2} and \eqref{dynamics:DS_3} on $[t_0,t_0+T]$ for some $T>0$;
		\item if $x_0\in \partial \hat{\setD}$ and assumption \ref{assumption:A2} holds, then the solution of \eqref{dynamics:DS_2} is unique on $[t_0,t_0+T_1]$ for some $T_1>0$ and the solution of \eqref{dynamics:DS_3} exists on $[t_0,t_0+T_2]$ for some $T_2>0$.
	\end{enumerate}
\end{corollary}
\begin{proof}
	The proof will follow the lines of the proof of Proposition~\ref{contraction2} up to Step 3 by replacing $\tilde{F}$ with $F$ and then we proceed as follows.
	
	We consider the following two cases.
	\begin{itemize}
		\item[Case 1.] Suppose $x_0\in \hat{D}$ such that $\rho:=\inf\limits_{y\in \partial\hat{D}}\|x_0-y\|_{H}=:{\rm dist}(x_0,\partial\hat{D})>0$.
		\item[Case 2.] Suppose $x_0\in \hat{D}$ such that $\rho=0$. Then one can follow the proof of Proposition~\ref{contraction2}.
	\end{itemize}
	We look for a solution to~\eqref{dynamics:DS_2} for Case 1.
	Let us consider
	\begin{align*}
	C_{0D}:=\left\{x(t)\in  C_0:\quad|x-x_0|_{C_{0}}\equiv \max_{t\in[t_0,t_0+T]}\|x(t)-x_0 \|_\hilbertX\le \rho/2,\ x_0\in \hat{D} \right\}.
	\end{align*}
	\bigskip
	Given the fact that $x_0\in \hat{D}:=\{ x\in D \mid \|x-\bar{w}\|^2\geq r >0 \}$, it implies $\rho:=||x_0 -\bar{w}||>0$.
	Let us consider the following two possible cases for fixed $r>0$,
	\begin{itemize}
		\item[(i)] if $\rho>2r$, then consider the ball $B_{\frac{\rho}{2}}(x_0)\subset \hat{D}$;
		\item[(ii)] if $\rho<2r$, then consider the ball $B_{r-\frac{\rho}{2}}(x_0)\subset \hat{D}$.
	\end{itemize}
	Thereafter, as in Step 4 of Proposition~\ref{contraction2} we  show the existence of Cauchy sequence in $C_{0D}$.
	\begin{itemize}
		\item[Step 5.] Moreover, $C_{0D}$ is a closed subset of $C_0$. Indeed, it is an implication of the facts of continuity of $S$ and 
		\begin{align*}
		x_n\in D\Rightarrow \lim\limits_{n\to\infty}x_n=:\hat{x}\in D,\ \text{since}\ D\ \text{is closed in}\ H.
		\end{align*}
		
		\item[Step 6.] Finally, $D\ni\hat{x}$ must be a fixed point of $S:C_{0D}\to C_{0D}$. Indeed,
		\begin{align*}
		\hat{x}=\lim\limits_{n\to\infty}x_n=\lim\limits_{n\to\infty}S[x_{n-1}]\stackrel{\text{continuity of}\ S}{=}S\left[\lim\limits_{n\to\infty}x_{n-1}\right]=S[\hat{x}].
		\end{align*}
	\end{itemize}
	Hence, we reach at the solution to~\eqref{dynamics:DS_2}.  
\end{proof}

In the following example we show that the existence of solutions of \eqref{dynamics:DS_2} is not guaranteed without assumption \ref{assumption:A2}, however there are still solutions of \eqref{dynamics:DS_3} due to Proposition \ref{contraction2}.

\begin{example} 
	Let $\hilbertX=\mathbb{R}^2$, $\bar{w}=(-1,0)$, $\bar{z}=(1,0)$, $\hat{\setD}=\bar{B}((0,0),1)\setminus B((-1,0),1)$ and let $F:\ \hat{\setD}\rightarrow \hilbertX$ be defined as
	\begin{equation*}
	F((x_1,x_2))=(1-x_1,0),\quad (x_1,x_2)\in \hat{\setD}
	\end{equation*}
	Then assumption \ref{assumption:A1} and \ref{assumption:A3} is satisfied.
	Consider $x_0=x(0)=(0,-1)$. Then there is no solution of \eqref{dynamics:DS_2}. By extending $F(x)$ in a continuous way:
	\begin{equation*}
	\tilde{F}((x_1,x_2))=(1-x_1,0),\quad (x_1,x_2)\in \hilbertX,
	\end{equation*}
	we obtain that one solution of \eqref{dynamics:DS_3} is $x(t)=(1-e^{-t},-1)$.
\end{example}

The following example shows that by considering \eqref{dynamics:DS_3} under assumption \ref{assumption:A2} we may loose the uniqueness of  solutions in the sense of Definition \ref{definition1}.

\begin{example}\label{ex:nonunique}
	Let $\hilbertX=\mathbb{R}^2$, $\bar{w}=(0,-1)$, $\bar{z}=(1,0)$, $\hat{\setD}=[0,1]\times[-1,0]\setminus B((0,-1),1)$ and let $F:\ \hat{\setD}\rightarrow \hilbertX$ be defined as
	\begin{equation*}
	F((x_1,x_2))=(1-x_1,0-x_2),\quad (x_1,x_2)\in \hat{\setD}
	\end{equation*}
	Then assumptions \ref{assumption:A1}, \ref{assumption:A2} and \ref{assumption:A3} are satisfied.
	Consider $x_0=x(0)=(0,0)$.  By extending $F(x)$ in the continuous way:
	\begin{equation*}
	\tilde{F}((x_1,x_2))=\left\{\begin{array}{ll}
	(1-x_1,0-x_2)  & (x_1,x_2)\in \hat{\setD} , \\
	(1-x_1,x_1) & (x_1,x_2)\in \Gamma:=\{ (1-e^{-s},e^{-s}+s-1), s\in (0,1]\},\\
	\text{continuous} & \text{otherwise on } \hilbertX.
	\end{array}\right.
	\end{equation*}
	We obtain that there are more solutions than one of the system \eqref{dynamics:DS_3}. For example:
	\begin{align*}
	&(x_1(t),x_2(t))=(1-e^{-t},e^{-t}+t-1),\quad t\in[0,1],\\
	&(x_1(t),x_2(t))=(1-e^{-t},0),\quad t\in[0,1].
	\end{align*}
\end{example}

\section{Extendability of solutions to \texorpdfstring{\eqref{dynamics:DS}}{dynamicsDS2}}
\label{section:extendability}
In this section we prove Theorem \ref{proposition_prolongation}. 
The proof is based on two lemmas, Lemma \ref{lemma2}  and Lemma \ref{Lemma:limit_Cauchy} (see Appendix).   
The proposed approach follows the lines of Lecture 3 of the lecture notes  \cite{lectures}. The crucial assumptions are \ref{assumption:A1}, \ref{assumption:A2} and \ref{assumption:A3} (see Lemma \ref{lemma2} below). For more general results and examples on the extendability of solutions, see e.g., \cite{MR2034199} and the references therein.

Let
$$
{\mathcal T}=[t_{0};T), \ t_{0}<T\leq+\infty\ \  \text{or } {\mathcal T}=[t_{0};T],\ t_{0}<T<+\infty.
$$

As a consequence of the results of Proposition  \ref{proposition:existence_finite2} and Corollary \ref{lemma1} we have the following 'non-branching' result.

\begin{lemma}\label{lemma2}
	Suppose that assumptions \ref{assumption:A1}, \ref{assumption:A2} and \ref{assumption:A3} are satisfied. Let $x_1(t)$, $x_2(t)$ be solutions to problem \eqref{dynamics:DS_2}  in the sense of Definition \ref{definition1}
	on ${\mathcal T}_1$,  ${\mathcal T}_2$, respectively. Then one of these solutions is a prolongation of the other (in particular, they coincide if $\mathcal{T}_1=\mathcal{T}_2)$.
\end{lemma}
\begin{proof}
	On the contrary, suppose that
	\begin{equation*}
	x_1(t) \not\equiv x_2(t) \quad \text{on} \ \mathcal{T}_1\cap \mathcal{T}_2.
	\end{equation*}
	Consider the set
	\begin{equation*}
	\mathcal{T}^{\neq } := \{ t\in  \mathcal{T}_1\cap \mathcal{T}_2 \mid x_1(t)\neq x_2(t) \}.
	\end{equation*}
	Let us note that $t_0\notin \mathcal{T}^{\neq}$ (by initial condition of \eqref{dynamics:DS_2}). Furthermore, the set $\mathcal{T}^{\neq}$ is open in the set $\mathcal{T}_1\cap \mathcal{T}_2$, because it is an inverse image of $(t_0,+\infty)$ under continuous mapping $t\rightarrow \|x_1(t)-x_2(t)\|$ defined on $\mathcal{T}_1\cap \mathcal{T}_2$.
	
	Put
	\begin{equation*}
	T^*=\inf \mathcal{T}^{\neq}.
	\end{equation*}
	Let us note that $T^*\notin \mathcal{T}^{\neq}$ (hence $x_1(T^*)=x_2(T^*)$). Indeed, if $T^*=t_0$ then, $t_0\notin T^{\neq}$ because $x_1(t_0)=x_2(t_0)$.
	If $T^*>t_0$, then $T^*$ is a boundary point of $\mathcal{T}^{\neq}$, so $T^*\notin \mathcal{T}^{\neq}$  since $\mathcal{T}^{\neq}$ is open in $\mathcal{T}_1\cap \mathcal{T}_2$. This means that in any right-hand side half-neighbourhood\footnote{By the right-hand side half-neighbourhood of a given $t\in \mathbb{R}$ we mean an interval in a form $[t,\alpha)$ for any $\alpha>t$.} of the point $T^*$ there exists $t_1>T^*$ such that $t_1\in \mathcal{T}^{\neq}\subsetneq \mathcal{T}_1\cap \mathcal{T}_2$, and the intersection of this right-hand side half-neighbourhood with $\mathcal{T}^{\neq}$ is nonempty. 
	
	Take any $\alpha>T^{*}$ and   $t_1$,  $t_1\in \mathcal{T}^{\neq}\cap [T^{*},\alpha)$.
	By Remark \ref{remark:restriction}, functions $x_1(t)$, $x_2(t)$ are solutions to Cauchy problem
	\begin{equation}\label{forumla:11}
	\left\{\begin{array}{ll}
	\dot{x}(t)=F(x(t)), & t>T^*\\
	x(T^*)=x_1(T^*)
	\end{array}\right.
	\end{equation}
	on interval $[T^*,t_1]$. 
	Since  $x_1(t), x_2(t)\in\hat{\setD}$ for all $t\in [T^*,t_1]$ and the set $\hat\setD$ is bounded, we have
	\begin{equation}\label{formula:12}
	R_{12}=\max_{i=1,2} \sup_{t\in [T^*,t_1]} \|x_i(t)-x_1(T^*)\|< +\infty.
	\end{equation}
	By Corollary \ref{lemma1}
	, there exists $T^\prime>t_0$, such that for any $T\in (t_0,T^\prime]$, solution of the Cauchy problem \eqref{forumla:11} on interval $[T^*,T^*+T]$ satisfying
	\begin{equation}\label{formula:13}
	\|x(t)-x_1(T^*)\|\leq R_{12}
	\end{equation}	
	is unique. Taking $T=\min\{ T^\prime, t_1-T^*  \}$  
	we come to a contradiction with Corollary \ref{lemma1},  because, by \eqref{formula:12} the condition \eqref{formula:13} holds both for $x_1(t)$ and $x_2(t)$, but the functions  $x_1(t)$ and $x_2(t)$ are different in any right-hand side half-neighbourhood of  $T^*$.  
\end{proof}


Now we are ready to prove  Theorem~\ref{proposition_prolongation}.

\begin{proof}[of Theorem~\ref{proposition_prolongation}]
	By Corollary \ref{lemma1}, there exists solution of  problem \eqref{dynamics:DS}  on some interval $[t_0, T] $ ($T>t_0$) in the class $\mathbb{B}_{t_0,x_{0},T}^R$ for some $R>0$. By Lemma \ref{lemma2}, for any two solutions of our problem  \eqref{dynamics:DS}  on different intervals, one is the prolongation of the other. 
	
	Consider now, for any $T>t_0$, all functions from $C^1([t_0,T],\hat{\setD})$. Among these functions there exist solutions of  problem \eqref{dynamics:DS}  or not. Put
	\begin{align}
	\begin{aligned}
	&\mathbb{T}=\{ T>t_0\mid \exists \text{ solution to } \eqref{dynamics:DS} \text{ from }  C^1([t_0,T],\hat{\setD}) \},\\
	&T_0=\sup \mathbb{T}.
	\end{aligned}
	\end{align} 
	If $T_0=+\infty$, there exists  solution $\tilde{x}(t)\in  C^1([t_0,+\infty),\hat{\setD})$ to problem \eqref{dynamics:DS}. Indeed, by taking a monotone increasing sequence $T_n\to+\infty$ and the corresponding sequence of solutions $\{x_n(t)\}$, by Lemma \ref{lemma2} we get, for all $n\in \mathbb{N}$ solution $x_{n+1}$ is the prolongation of $x_n$. Hence, the function
	\begin{equation*}
	\tilde{x}(t)=\left\{ \begin{array}{lll}
	x_n(t), & t\in [T_{n-1},T_n), & n\geq 2\\
	x_1(t), & t\in [t_0,T_1)
	\end{array}\right.
	\end{equation*}
	is a solution defined on $[t_0,+\infty)$. Other solutions (which do not coincide with the restrictions of $\tilde{x}(t)$ on smaller intervals) do not exist by Lemma \ref{lemma2}. In the rest of the proof, we show that this is the only possible case.
	
	Consider now $T_0<+\infty$. Then two cases are possible:
	\begin{enumerate}[label=(\alph*)]
		\item \label{case:a} $T_0\in  \mathbb{T}$,
		\item \label{case:b} $T_0\notin  \mathbb{T}$.
	\end{enumerate}
	In case \ref{case:a} there exists a solution $x(\cdot)\in C^1([t_0,T_0],\hat{\setD})$ to problem \eqref{dynamics:DS}. But then, by Corollary \ref{lemma1}, applied to our problem  \eqref{dynamics:DS} with $t_0=T_0$ solution can be extended beyond $T_0$ 
	and both one-sided derivatives $\dot{x}_{-} (T_0)$ and $\dot{x}_{+} (T_0)$ exist and both equal $F(x(T_0))$: left - by the definition of solutions on $[t_0,T_0]$, right - by the definition of solution to our problem with the beginning of the interval from $T_0$. As a consequence, we get a solution on a larger interval and arrive to a contradiction with the definition of $T_0$. This excludes case \eqref{case:a}.
	
	In case \ref{case:b}, by the arguments analogous to the case $T_0=+\infty$, we get the existence and uniqueness of solutions $x(t)$ of \eqref{dynamics:DS}  on the semi-interval $[t_0,T_0)$. Case \ref{case:b} splits in two subcases:
	\begin{enumerate}
		\item \label{case:b1} $\limsup _{t\rightarrow T_0^{-}} \|x(t)\|=+\infty$ (i.e. solution is unbounded in any left-sided interval of $T_0$),
		\item \label{case:b2}  $\limsup _{t\rightarrow T_0^{-}} \|x(t)\|<+\infty$.
	\end{enumerate} 
	
	The subcase \ref{case:b1} is impossible in view of the boundedness of the set  $\hat \setD$. Now we show that the subcase \ref{case:b2} is also impossible.	
	Indeed, let the function $x(t)$ be bounded on the whole half-interval $[0,T_0)$:
	\begin{equation*}
	\exists C\geq 0 \ \forall t\in [t_0, T_0)\ \|x(t)\|\leq C.
	\end{equation*}
	We have
	\begin{equation*}
	\forall t\in  [t_0, T_0)\ \|F(x(t))\|\leq M.
	\end{equation*}
	However, from the equation \eqref{dynamics:DS}, it follows that 
	the function $x(t)$ is Lipschitz continuous with a constant $M$ on $(t_0,T_0)$, since $\|\dot x(t)\|\le M$ for all $t\in (t_0,T_0)$.
	Hence, by Lemma \ref{Lemma:limit_Cauchy} (see Appendix), there exists the limit
	\begin{equation*}
	Y_0=\lim_{t_\rightarrow T_0^-} x(t).
	\end{equation*}
	
	Let us put $Y_0$ to be the value of $x(t)$ at $T_0$. The obtained function $Y(t)$ will be continuous from the left at $T_0$. Then, by Lemma \ref{lemma:continuous} (see Appendix),  the function $F(Y(T_0))$ is also continuous  from the left at $T_0$ and hence
	\begin{equation*}
	\lim_{t_\rightarrow T_0^-} F(x(t))=\lim_{t_\rightarrow T_0^-} F(Y(t))= F(Y_0).
	\end{equation*}
	Since for $t<T_0$ we have $\dot{x}(t)=F(x(t))$, from the last formula we get
	\begin{equation*}
	\lim_{t_\rightarrow T_0^-} \dot{x}(t)=F(Y_0).
	\end{equation*}
	However, by Lemma about extendability at point (Lemma \ref{lemma:prolongation}, see Appendix),  it follows that the function $x(t)$ can be extended from $[t_0,T_0)$ onto $[t_0,T_0]$ with preservation of continuous differentiability (let us denote the obtained function by $Y(t)$) and $\dot{Y}(t_0)=F(Y_0)$ and $Y(t)$ is a solution on $[t_0,T_0]$. We arrive to a contradiction in the subcase \ref{case:b2} of case \ref{case:b}  (solutions on $[t_0, T_0]$ do not exist).   
\end{proof}

\section{Behaviour of trajectories at \texorpdfstring{$+\infty$}{infty}}\label{section:behaviour} 
In this section we prove Theorem \ref{proposition:conditions_convergence} and provide other results concerning the convergence of trajectories.


\begin{proposition}\label{proposition:sequence_to_convergence}
	Let $x(t)$, $t\in [t_0,+\infty)$ be a solution of \eqref{dynamics:DS}. 
	Suppose that there exists an increasing sequence  $\{t_n\}_{n\in \mathbb{N}}$, $t_n\rightarrow +\infty$, such that $x(t_n)\rightarrow \bar{z}$. Then $x(t)\rightarrow \bar{z}$ as $t\rightarrow +\infty$.
\end{proposition}

\begin{proof}
	Let $\{t_n\}_{n\in \mathbb{N}}$, $t_n\rightarrow +\infty$ be such that $x(t_n)\rightarrow \bar{z}$. We will show that for all $\varepsilon>0$, for every  increasing sequence $\{s_n\}_{n\in \mathbb{N}}$, $s_n\rightarrow +\infty$ there exists $n_0\in \mathbb{N}$ such that for all $n\geq n_0$, $\|x(s_n)-\bar{z}\|\leq \varepsilon$. Take any $\varepsilon>0$ and an increasing sequence $\{s_n\}_{n\in \mathbb{N}}$, $s_n\rightarrow +\infty$.
	
	We have
	\begin{equation*}
	\frac{d}{dt}\|x(t)-\bar{w}\|^2= 2\langle F(x(t)) \mid x(t)-\bar{w}\rangle\geq 0, 
	\end{equation*}
	hence the function $\|x(\cdot)-\bar{w}\|^2$ is nondecreasing.
	Moreover, by \eqref{inequality:setD} (see also Lemma~\ref{inequaltiy:triangle_ball}) and convergence of $x(t_n)$, for all $\varepsilon^\prime>0$ there exists $n_0^\prime\in \mathbb{N}$ such that for all $n>n_0^\prime$
	\begin{equation*}
	\|x(t_n)-\bar{w}\|^2\geq \|\bar{w}-\bar{z}\|^2 -\varepsilon^\prime.
	\end{equation*}
	Take $\varepsilon^\prime=\varepsilon$ and $n_0$ such that $s_{n_0}\geq t_{n_0^\prime}$.
	
	Then, by \eqref{inequality:setD}
	and the fact that $\|x(\cdot)-\bar{w}\|^2$ is nondecreasing we obtain: for all $n\geq n_0$
	\begin{align*}
	\|x(s_n)-\bar{z}\|^2\leq \|\bar{w}-\bar{z}\|^2-\|x(s_n)-\bar{w}\|^2\leq \|\bar{w}-\bar{z}\|^2-\|x(t_{n_0}^\prime)-\bar{w}\|^2\leq \varepsilon.\  
	\end{align*}
\end{proof}

Now we give now the proof of Theorem~\ref{proposition:conditions_convergence}.
\begin{proof}[of Theorem~\ref{proposition:conditions_convergence}].
	By \eqref{assumption:weak2}, we have $\tilde x=\bar z$, i.e., $ x(t_{n_k})$ converges weakly to $\bar z$. By  
	\eqref{inequality:setD}
	, the following inequality holds for this subsequence
	$$\|x(t_{n_k})-\bar w\|^2+ \|x(t_{n_k})-\bar z\|^2\leq \|\bar w-\bar z\|^2,\quad k=1,2,\ldots $$
	and hence 
	$$\liminf_{k\to\infty} \|x(t_{n_k})-\bar w\|^2 +\liminf_{k\to\infty} \|x(t_{n_k})-\bar z\|^2  \le \|\bar w-\bar z\|^2. \eqno{(*)}$$ 
	Since the norm is weakly  lower semicontinuous,  we also have 
	$$\|\bar z-\bar w\|^2\le\liminf_{k\to\infty} \|x(t_{n_k})-\bar w\|^2.   
	$$
	This and $(*)$  implies
	$$ \liminf_{k\to\infty}\|x(t_{n_k})-\bar z\|^2= 0.$$
	Consequently, there is a subsequence $t_{n_{k_m}}$ such that 
	$$\lim_{m\to\infty} \|x(t_{n_{k_m}})-\bar z\|=0.$$
	Thus we have shown that for any sequence $\{t_n\}_{n\in \mathbb{N}}$, $t_n\to\infty$, there exists a subsequence  $\{t_{n_{k_m}}\}_{m\in \mathbb{N}}$ such that the above condition holds. 
	This means that  $\|x(t)-\bar z\|\to 0$ as $t\to +\infty$.  
\end{proof}

In the next two propositions we propose variants of Theorem \ref{proposition:conditions_convergence}  in which we replace assumption \eqref{assumption:weak2} by other assumptions.

In the finite-dimensional case, the assertion of Theorem \ref{proposition:conditions_convergence}  can be obtained without assuming \eqref{assumption:weak2}. Instead we need to assume a strengthened form \eqref{cond:Cprime} of the assumption \ref{assumption:A3} on vector field $F$.

Recall that the assumption \ref{assumption:A3} says that
$\langle F(x) \mid \bar{w} - x \rangle \leq 0$ for all $x\in \hat{\setD}$.
\begin{proposition}
	\label{proposition 55}
	Let $\hilbertX$ be a  finite-dimensional space, let $x(t)$, $t\in [t_0,+\infty)$ be a solution of \eqref{dynamics:DS} and assume that 
	\begin{equation}\tag{C*}\label{cond:Cprime}
	\langle F(x(t)) \mid \bar{w}-x(t) \rangle <0 \quad \forall\ t\in [t_0,+\infty) \  x(t)\neq\bar{z} . 
	\end{equation}
	Then $\lim\limits_{t\rightarrow +\infty} x(t) =\bar{z}$.
\end{proposition}
\begin{proof}
	Let $g(t):=\frac{d}{dt}\|x(t)-\bar{w}\|^{2}$, $t\ge t_{0}.$
	We start by showing that there exists a sequence $\{t_k\}$, $t_k\rightarrow +\infty$ such that $\lim\limits_{k\rightarrow +\infty} g(t_k)= 0$. 
	
	On the contrary, suppose that there exist $\varepsilon>0$ and  $t^\prime \geq t_0$ such that  $g(t)>\varepsilon$ for all $t>t^\prime$. Hence, for all $t> t^\prime$
	\begin{align*}
	&\|x(t)-\bar{w}\|^2 - \|x(t_0)-\bar{w}\|^2=\int_{t_0}^{t} g(s) \, ds \\
	&= \int_{t^0}^{t^\prime} g(s) \, ds + \int_{t^\prime}^{t} g(s) \, ds \geq \int_{t^0}^{t^\prime} g(s) \, ds  + \int_{t^\prime}^{t} \varepsilon \, ds \\
	&=\int_{t^0}^{t^\prime} g(s) \, ds + (t-t^\prime) \varepsilon.
	\end{align*}
	By taking 
	\begin{equation*}
	t>\frac{1}{\varepsilon}\left(\|\bar{z}-\bar{w}\|^2-\|x(t_0)-\bar{w}\|^2- \int_{t^0}^{t^\prime} g(s) \, ds\right)+t^\prime
	\end{equation*}
	we arrive to $\|x(t)-\bar{w}\|^2>\|\bar{z}-\bar{w}\|$, i.e. $x(t)\notin \hat{\setD}$ - a contradiction. In this way, we proved that there exists a sequence  $\{t_k\}_{k\in \mathbb{N}}$ such that $t_k\rightarrow +\infty$ and $\lim\limits_{k\rightarrow +\infty} g(t_k)= 0$. 
	
	Since $\hilbertX$ is finite-dimensional and $\hat{\setD}$ is closed, bounded, hence compact. 
	There exists 
	a subsequence of $\{t_k\}_{k\in \mathbb{N}}$, namely $\{ t_{k_n} \}_{n\in \mathbb{N}}$ such that $x(t_{k_n})$ converges 
	and $\lim\limits_{n\rightarrow +\infty } x(t_{k_n})=\tilde{x}\in\hat{\setD}$. Without loss of generality, we may assume that the sequence $\{t_{k_n} \}_{n\in \mathbb{N}}$ is  increasing. 
	
	By  \ref{fact:subset_ball},
	\begin{equation*}
	\frac{d}{dt} \|x(t)-\bar{w}\|^2 = 2\langle F(x(t)) \mid x(t) - \bar{w} \rangle \geq 0\quad \text{for all } t\geq t_0.
	\end{equation*}
	We have 
	\begin{equation}\label{convergence:remark_weak2strong}
	0=\lim_{n\rightarrow +\infty } g(t_{k_n}) = \lim_{n\rightarrow +\infty } 2 \langle F(x(t_{k_n})) \mid x(t_{k_n})-\bar{w} \rangle = 2 \langle F(\tilde{x}) \mid \tilde{x}-\bar{w} \rangle,
	\end{equation}
	hence, by assumption, $\tilde{x}=\bar{z}$. Now the assertion  follows from Proposition \ref{proposition:sequence_to_convergence}.   
	
\end{proof}

\begin{remark} 
	By examining the above proof, we see that  the assertion of  Proposition  \ref{proposition 55},  remains true in {\em infinite-dimensional} Hilbert space $\hilbertX$ 
	under additional assumption (to \eqref{cond:Cprime}) on $F$:
	\begin{align}\tag{W-S}
	\begin{aligned}
	&F \text{ can be extended to  } \operatorname{conv}\, \hat{\setD}  \text{ in such way that }\\
	&F:\ \operatorname{conv}\, \hat{\setD}\rightarrow \hilbertX \text{ is a weak-to-strong continuous on } \operatorname{conv}\, \hat{\setD},
	\end{aligned}
	\end{align}
	i.e., for any weakly convergent sequence  $\hat{\setD}\ni x_{n}\rightharpoonup\bar{x}$ we have $\lim\limits_{n\rightarrow+\infty} F(x_n)=F(\bar{x})$, where the limit is strong.
	
	The need of using this additional assumption  follows from the fact that if $v_{n}\rightarrow v$ and  $u_{n}\rightharpoonup u$ 
	, then $\langle v_{n}\mid u_{n}\rangle\rightarrow \langle v\mid u\rangle$. Indeed, 
	$$
	|\langle v_{n}\mid u_{n}\rangle-\langle v\mid u\rangle|\begin{array}[t]{l}
	=|\langle v_{n}-v\mid u_{n}\rangle+\langle u_{n}\mid v\rangle-\langle v\mid u\rangle|\\
	\le\|u_{n}\|\|v_{n}-v\|+|\langle u_{n}\mid v\rangle-\langle v\mid u\rangle|.
	\end{array}
	$$
	This fact allows to show \eqref{convergence:remark_weak2strong}.
\end{remark}


The following proposition is a variant of Proposition \ref{proposition 55} valid in infinite-dimensional Hilbert space under a more restrictive form \eqref {cond:Cprimeprime} below of condition \eqref{cond:Cprime}.

\begin{proposition}
	Let $\hilbertX$ be an infinite-dimensional space and let $x(t)$, $t\in [t_0,+\infty)$ be a solution of \eqref{dynamics:DS}. 
	Assume that for all $t\in [t_0,+\infty)$ such that $x(t)\neq\bar{z}$, we have 
	\begin{equation}\tag{C**}\label{cond:Cprimeprime}
	\langle F(x(t)) \mid \bar{w}-x(t) \rangle  <\alpha(t),
	\end{equation} where $\alpha:[t_{0},+\infty) \rightarrow \mathbb{R}_{-}$ is an integrable function on any interval $[t_{0},T]$, $T>t_0$ 
	and  there exist $T'>t_0$, $\sqrt{T^\prime}-\frac{t_0}{\sqrt{T^\prime}}> \frac{1}{2}(\|\bar{w}-\bar{z}\|^2-\|x(t_0)-\bar{w}\|^2)$ 
	and $\varepsilon\leq \frac{-1}{\sqrt{T^\prime}}$
	such that  $\sup\limits_{[t_{0},T'] } \alpha(s)<\varepsilon
	$.
	Then $\lim\limits_{t\rightarrow +\infty} x(t) =\bar{z}$.
\end{proposition}
\begin{proof}
	Let us note that in the case when there exists $t^\prime\in [t_0,+\infty)$ such that $x(t^\prime)=\bar{z}$, then $x(t)=\bar{z}$ for all $t>t^\prime$ since $F(x(t^\prime))=F(\bar{z})=0$.
	
	Consider now the situation that $x(t)\neq \bar{z}$ for any $t\in [t_0,+\infty)$. 
	By contradiction, suppose that $x(t)\not\rightarrow \bar{z}$. Then, in view of Proposition \ref{proposition:sequence_to_convergence}, there exists $\varepsilon>0$ such that $x(t)\notin B(\bar{z},\varepsilon)$ for all $t\in [t_0,+\infty)$. 
	
	We have that for all $t>T'$
	\begin{align*}
	&\|x(t)-\bar{w}\|^2-\|x(t_0)-\bar{w}\|^2 = \int_{t_0}^{t} \frac{d}{ds} \|x(s)-\bar{w}\|^2\, ds \\
	&= 2\int_{t_0}^{t} \langle F(x(s)) \mid x(s)-\bar{w}\rangle  \, ds \geq -2\int_{t_0}^{t} \alpha(s)\, ds \\
	&\geq -2(t-t_0)\cdot  \sup_{s\in [t_0,t]} \alpha(s)\\
	&\geq -2(t-t_0)\cdot \varepsilon.
	\end{align*}
	Thereby for such $t> \frac{\|\bar{w}-\bar{z}\|^2}{2c}+t_0 $ we arrive to a contradiction with $x(t)\in \hat{\setD}\subset \setD$.  
\end{proof}


\begin{proposition}  
	Let $\hilbertX$ be an infinite-dimensional space and let $x(t)$, $t\in [t_0,+\infty)$ be a solution of \eqref{dynamics:DS}. Assume that for all $\varepsilon$ such that $0<\varepsilon<\|x_0-\bar{z}\|$ we have 
	$\inf\limits_{x \in \hat{\setD} \setminus B(\bar{z},\varepsilon)}\langle F(x) \mid \bar{w}-x \rangle  <0$. Then $\lim\limits_{t\rightarrow +\infty} x(t) =\bar{z}$.
\end{proposition}
\begin{proof}
	If there exists $t'\in [0,+\infty)$ such that $\langle F(x(t')) \mid w-x(t')\rangle =0$ then we are done - in view of assumptions of the Proposition, $x(t^\prime)=\bar{z}$, and by 
	\eqref{inequality:setD}
	, Proposition \ref{proposition:increasing}, $x(t)=\bar{z}$ for all $t\geq t'$.
	
	Suppose that for all $t\in [0,+\infty)$ we have $\langle F(x(t)) \mid w-x(t)\rangle <0$.	For any $t>t_0$
	\begin{align*}
	&\|\bar{z}-\bar{w}\|^2-\|x(t_0)-\bar{w}\|^2\geq \|x(t)-\bar{w}\|^2-\|x(t_0)-\bar{w}\|^2 = \int_{t_0}^t \frac{d}{ds} \|x(s)-\bar{w}\|^2\, ds \\
	&= 2\int_{t_0}^t \langle F(x(s)) \mid x(s)-\bar{w}\rangle  \, ds \geq -2(t-t_0)\cdot  \sup_{s\in [t_0,T]} \langle F(x(s)) \mid x(s)-\bar{w}\rangle \\
	&=2(t-t_0)\cdot  \inf_{s\in [t_0,T]} \langle F(x(s)) \mid w-x(s)\rangle
	\end{align*}
	Therefore $\inf\limits_{s\in [t_0,t]} \langle F(x(s)) \mid w-x(s)\rangle \rightarrow 0$ as $t\rightarrow +\infty$. Note that $\alpha(s):=\langle F(x(s)) \mid w-x(s)\rangle$ is a continuous function on every $[t_0,t]$, $t>t_0$. 
	Hence, there exists an increasing sequence $\{t_n\}_{n\in \mathbb{R}_+}$, $t_n\rightarrow +\infty$ such that $ \langle F(x(t_n)) \mid x(t_n)-\bar{w}\rangle \rightarrow 0$. We claim that $x(t_n)\rightarrow \bar{z}$.
	
	Suppose on the contrary, that $x(t_n)\not \rightarrow \bar{z}$ as $n\rightarrow +\infty$. Then there exists $\varepsilon>0$ and a subsequence $\{t_{n_k}\}_{k\in \mathbb{N}}$ such that $x(t_{n_k})\in \hat{\setD} \setminus B(\bar{z},\varepsilon)$. Since $\inf\limits_{x \in \hat{\setD} \setminus B(\bar{z},\varepsilon)}\langle F(x) \mid \bar{w}-x \rangle  <0$ we have that there exists $c<0$ such that $\langle F(x(t_{k_n})) \mid \bar{w}-x(t_{k_n}) \rangle  <c$, a contradiction to $\lim\limits _{k\rightarrow +\infty}\langle F(x(t_{k_n})) \mid \bar{w}-x(t_{k_n}) \rangle=0$ .	
	
	Hence $x(t_n)\rightarrow \bar{z}$. Now the assertion  follows from Proposition \ref{proposition:sequence_to_convergence}.   
\end{proof}

\section{Projective dynamical system}\label{section:PDS}

In this section, we give an example of the system \eqref{dynamics:DS}. Let $\bar{w}, \bar{z} \in \hilbertX$. We consider the projective dynamical system 
\begin{align}\tag{PDS}\label{system:PDS}
\begin{aligned}
&\dot{x}(t)=P_{\multifC(x(t))}(\bar{w})-x(t),\\
&x(t_0)=x_0\in \hat{\setD},\ t_0\geq 0,
\end{aligned}
\end{align}
where $\multifC:\ \hatD \rightrightarrows \hilbertX$ is a multifunction such that:
\begin{enumerate}[label=(\Alph*$^\prime$)]
	\item\label{assumption:A1:proj} for all $x\in \hatD$, $\bar{z}\in \multifC(x)$ and  $P_{\multifC(x)}(\bar{w})=x$ iff $x=\bar{z}$, 
	\item\label{assumption:A2:proj} for all $x\in \hat{\setD}$ we have $P_{\multifC(x)}(\bar{w})\in \setD$,
	\setcounter{enumi}{2}
	\item \label{assumption:A3:proj} $\langle P_{\multifC(x)}(\bar{w})-x \mid \bar{w} - x \rangle \leq 0$  for all $x\in \hatD$,
	\item\label{assumption:A4:proj} for all $x\in \hatD$, $\multifC(x)$ is closed and convex.
\end{enumerate}

Condition \ref{assumption:A4:proj} ensures that the projection onto $\multifC(x)$, $x\in \hat{\setD}$ is uniquely defined. 

The condition  $\langle P_{\multifC(x)}(\bar{w})-x \mid \bar{w} - x \rangle \leq 0$  for all $x\in \hatD$ is equivalent to the condition that $P_{\multifC(x)}(\bar{w}) \in H(\bar{w},x)$ for any $x\in \hatD$. This implies that for any $x\in \hatD$ and for any $h\in \multifC(x)$ we have $\langle h-x \mid \bar{w} - x \rangle \leq 0$. The later implies  $P_{\multifC(x)}(\bar{w}) \in H(\bar{w},x)$. Therefore,  \ref{assumption:A3:proj} is equivalent to the condition:
\begin{equation*}
\forall x\in \hatD\ \forall h\in \multifC(x), \quad \langle h-x \mid \bar{w}-x \rangle \leq 0. 
\end{equation*}


\begin{remark}
	Let us comment on the conditions \ref{assumption:A1:proj},  \ref{assumption:A2:proj}, \ref{assumption:A3:proj}. The condition \ref{assumption:A1:proj} is equivalent to saying that $\bar{z}$ is the only stationary point of the vector field $F(x)=P_{\multifC(x)}(\bar{w})-x$ inside the considered set $\hatD$. 
	The condition \ref{assumption:A2:proj} together with the convexity of set $\setD$ ensures that for any $\lambda\in [0,1]$, and for any $x\in \hatD\subset \setD$ it is $(1-\lambda) x + \lambda P_{\multifC(x)}(\bar{w}) \in \setD$. 
	The condition \ref{assumption:A3:proj} ensures that $P_{\multifC(x)}\in H(\bar{w},x)$ and the  function $t\mapsto \|x(t)-\bar{w}\|$ is nondecreasing (see e.g., Proposition \ref{proposition:increasing}), where $x(t)$ is a solution of \eqref{system:PDS} (whenever it exists). 
\end{remark}

As  consequences of Theorem \ref{proposition_prolongation} and Theorem \ref{proposition:conditions_convergence} we can formulate the following theorems.
\begin{theorem}\label{theorem:PDSexistence}
	Suppose that \ref{assumption:A1:proj}, \ref{assumption:A2:proj}, \ref{assumption:A3:proj}, \ref{assumption:A4:proj} holds. Assume that $x\mapsto P_{\multifC(x)}(\bar{w})$ is  locally Lipschitz continuous on $\hat{\setD}\setminus \{\bar{z}\}$ and continuous on $\hatD$. Then the system \eqref{system:PDS} has a unique solution on $[t_0,+\infty)$.
\end{theorem}
\begin{proof}
	First, let us show that \ref{assumption:A1}, \ref{assumption:A2}, \ref{assumption:A3} hold. \ref{assumption:A1:proj} implies that $\bar{z}$ is the only stationary point of \eqref{system:PDS}, hence \ref{assumption:A1} holds.
	
	Recall that $\setD$ is a closed, convex subset of $\mathbb{B}(\frac{\bar{w}+\bar{z}}{2},\frac{\|\bar{w}-\bar{z}\|}{2})$ and  $\hatD$ is given as in \eqref{set:Dhat}. 
	By \ref{assumption:A4:proj}, the projection $P_{\multifC(x)}(\bar{w})$ is well defined for all $x\in \hatD$. By \ref{assumption:A2:proj} and  \ref{assumption:A3:proj},  assumption \ref{assumption:A2} is satisfied since for all $x\in \hat{\setD}\subset \setD$ and for any $h\in [0,1]$
	\begin{align*}
	& x+h(P_{\multifC(x)}(\bar{w})-x)=(1-h)x + h P_{\multifC(x)}(\bar{w}) \in \setD,\\
	& \|x+h(P_{\multifC(x)}(\bar{w})-x)-\bar{w}\|^2=\|x-\bar{w}\|^2 \\
	& - 2h \langle P_{\multifC(x)}(\bar{w})-x \mid \bar{w}-x \rangle + h^2 \|P_{\multifC(x)}(\bar{w})-x\|^2\geq \|x-\bar{w}\|^2\geq r ,
	\end{align*}
	i.e. $x+h(P_{\multifC(x)}(\bar{w})-x)\in \hat{\setD}$. Note that by taking $h=1$ we obtain that $P_{\multifC(x)}(\bar{w})\in \hat{\setD}$ for any $x\in \hatD$. 
	Assumption \ref{assumption:A3:proj} is equivalent to \ref{assumption:A3} for $F(x)=P_{\multifC(x)}(\bar{w})-x$. Observe that the mapping $F(x)=P_{\multifC(x)}(\bar{w})-x$ is bounded on $\hatD$. Indeed for any $x\in \hatD$ we have
	\begin{equation*}
	\|P_{\multifC(x)}(\bar{w})-x\|\leq \|P_{\multifC(x)}(\bar{w})\| + \|x\|\leq 2R,
	\end{equation*}
	where $R=\sup_{x\in \hatD} \|x\|$. Now, system \eqref{system:PDS} is of the form of \eqref{dynamics:DS} with $F(x)=P_{\multifC(x)}(\bar{w})-x$ and all the assumptions of Theorem \ref{proposition_prolongation} are satisfied.	The assertion of the theorem follows from Theorem \ref{proposition_prolongation}.  
\end{proof}

\begin{theorem}
	Suppose that \ref{assumption:A1:proj}, \ref{assumption:A2:proj}, \ref{assumption:A3:proj}, \ref{assumption:A4:proj} holds. Assume that $x\mapsto P_{\multifC(x)}(\bar{w})$ is  locally Lipschitz continuous on $\hat{\setD}\setminus \{\bar{z}\}$ and continuous on $\hatD$. Let $x(t)$ be a solution of \eqref{system:PDS}.  Assume that	for every increasing sequence $\{t_n\}_{n\in\mathbb{N}}$, $t_n\rightarrow +\infty$
	\begin{equation}\label{assumption:weakPDS}
	x(t_{n})\rightharpoonup \tilde{x} \implies \tilde{x}=\bar{z},
	\end{equation}
	Then $x(t)\rightarrow \bar{z}$ as $t\rightarrow +\infty$.
\end{theorem}
\begin{proof}
	By the proof of Theorem \ref{theorem:PDSexistence},  \eqref{system:PDS}, assumptions \ref{assumption:A1}, \ref{assumption:A2} and \ref{assumption:A3} are satisfied, and by assumption $F(x)=P_{\multifC(x)}(\bar{w})$ is locally Lipschitz continuous. Now the assertion follows from Theorem  \ref{proposition:conditions_convergence}.
\end{proof}

To investigate the local Lipschitzness of $x \mapsto P_{\multifC(x)}(\bar{w})$   on $\hat{\setD}\setminus\{\bar{z}\}$ (and the continuity of  $x\mapsto P_{\multifC(x)}(\bar{w})$   on $\hat{\setD}$) one should take into account the  form of  multifunction $\multifC$. 	Behaviour of the projection of a given $\bar{w}$ onto polyhedral multifunction $\multifC$ given by a finite number of linear inequalities and equalities were investigated in e.g.  \cite[Corollary 2]{Bednarczuk2020}, see also \cite[Theorem 6.5]{full_stability_of_general_parametric_variational_systems}.

\begin{proposition}\label{proposition:opertPDS}
	Let $\opert:\ \hilbertX \rightarrow \hilbertX$, which appears in  system \eqref{system:Haugazeau},  
	be a firmly quasinonexpansive operator, i.e.,
	\begin{equation*}
	\forall x \in \hilbertX \ \forall y \in \operatorname{Fix}\, \opert ,\quad \|\opert x-y \|^2 + \|\opert x - x \|^2 \leq \|x-y\|^2.
	\end{equation*}
	Assume that $\bar{w}\in \hilbertX$, $\bar{w}\notin \operatorname{Fix}\, \opert$ and $\bar{z}=P_{\operatorname{Fix}\, \opert}(\bar{w})$,  and let $\setD=\mathbb{B}(\frac{\bar{w}+\bar{z}}{2}, \frac{\|\bar{w}-\bar{z}\|}{2})$, $\hatD=\{ x \in \setD \mid \|x-\bar{w}\|^2\geq r  \}$ for some $r\in (0,\|\bar{w}-\bar{z}\|^2)$.
	Then the assumptions \ref{assumption:A1:proj}, \ref{assumption:A2:proj},  \ref{assumption:A3:proj}, \ref{assumption:A4:proj} holds for the system \eqref{system:PDS} with $\multifC:\ \setD \rightarrow \hilbertX$ defined as  $\multifC(x):=H(\bar{w},x)\cap H(x,\opert x)$.
\end{proposition}
\begin{proof}

	By \cite[Corollary 4.25]{MR3616647}, we have 
	\begin{align}\label{assumption:A1:proj:fquasinonexpansive}
	\begin{aligned}
	\operatorname{Fix}\, \opert &= \bigcap_{ x \in \hilbertX} \{ y \in \hilbertX \mid \langle y - \opert x \mid x-\opert x\rangle \leq 0      
	&=\bigcap_{ x \in \hilbertX} H(x,\opert x).
	\end{aligned}
	\end{align}
	
	Assumption \ref{assumption:A1:proj} follows from \eqref{assumption:A1:proj:fquasinonexpansive} i.e., $\operatorname{Fix}\, \opert \ni \bar{z}\in H(x,\opert x)$ for all $x\in \hatD$ and $x\in \operatorname{Fix}\, \opert \cap \hatD \iff x = \bar{z}$.
	Assumption \ref{assumption:A2:proj} follows from fact that for any $x\in \hatD\subset \setD$, $\bar{z}\in \multifC(x)$, hence $\bar{z}\in H(\bar{w},P_{\multifC(x)}(\bar{w}))$ and therefore $P_{\multifC(x)}(\bar{w})\in \mathbb{B}(\frac{\bar{w}+\bar{z}}{2}, \frac{\|\bar{w}-\bar{z}\|}{2})=\setD$ (see  \eqref{fact:subset_ball}). Assumption \ref{assumption:A3:proj} follows from fact that for any $x\in \setD$, $\multifC(x)\subset H(\bar{w},x)$, hence
	$P_{\multifC(x)}(\bar{w})\in H(\bar{w},x)$. Assumption \ref{assumption:A4:proj} is satisfied since $H(\bar{w},x)\cap H(x,\opert x)$ is closed, convex.  \end{proof}
Depending upon the choice of the operator $\opert$ in Proposition \ref{proposition:opertPDS} we obtain dynamical systems of the form \eqref{system:PDS} related to different algorithms. 
Within our  approach we encompass the following dynamical systems related to the following algorithms.

\begin{enumerate}[label={Ex \arabic*.}]
	\item When $\opert:\ \hilbertX \rightarrow \hilbertX$ is firmly quasinonexpansive and $(Id-\opert)$ is demiclosed at $0$,   dynamical system \eqref{dynamics:DS} corresponds to the best approximation algorithm for finding a point $\bar{z}$ from the set of fixed points of $\opert$, i.e.,
	for finding $\bar{z}\in \hilbertX$ such that 
	$\bar{z}=P_{\operatorname{Fix}\opert}(\bar{w})$ (see \cite[Theorem  30.8]{MR3616647}).
	
	\item When $\opert=J_{A}$, where $A:\hilbertX \rightrightarrows \hilbertX$ is maximally monotone,  dynamical system \eqref{dynamics:DS} corresponds   to the best approximation algorithm for finding $x\in \hilbertX$ such that $0\in Ax$ (see \cite[Corollary 30.11]{MR3616647}). Let us recall that resolvent operator of $A$ is defined as  $J_{A}:\ \hilbertX \rightarrow \hilbertX$,
	$J_{A}=(Id-A)^{-1}$.
	
	\item  When $\opert=(1/2)(Id+J_{\gamma A}\circ (Id-\gamma B))$, $A:\hilbertX \rightrightarrows \hilbertX$ is maximally monotone, $B:\ \hilbertX \rightarrow \hilbertX$ is $\beta$-cocoercive, $\gamma \in [0,2\beta]$,  dynamical system \eqref{dynamics:DS} corresponds   to
	the best approximation algorithm  for finding  $x\in \hilbertX$ such that  $0\in Ax+Bx$ (see \cite[Corollary 30.12]{MR3616647}).
	
	\item When $\opert:\hilbertH\times\hilbertG\rightarrow \hilbertH\times\hilbertG$ is defined as 
	\begin{align} 
	\label{operatort} 
	\begin{aligned}
	\opert(x)=P_{H(x)}(x),\  H(x):=\{ h \in \hilbertH\times \hilbertG \mid \langle  h \mid s^*(x)\rangle \leq \eta(x) \},
	\end{aligned}
	\end{align}
	and, for any 
	$x=(p,v^*)\in \hilbertH\times \hilbertG$,
	\begin{align}\label{block}
	\left.
	\begin{aligned}
	& s^*(x) := (a^*(x) + L^*b^*(x), b(x) - La(x)); \\ 
	&\eta(x) :=\langle a(x) \mid a^*(x) \rangle + \langle b(x) \mid b^*(x) \rangle ;\\
	& a(x) := J_{\gamma A} (p - \gamma L^* v^*),\quad b(x) := J_{\mu B} (L p + \mu v^*) ;  \\
	& a^*(x) := {}^{\gamma}A(p - \gamma L^* v^*),\quad b^*(x) := {}^{\mu}{B}(Lp+ \mu v^*),\ \gamma,\mu \in (0,1),
	\end{aligned}
	\right\}
	\end{align}
	%
	dynamical system \eqref{dynamics:DS} corresponds   to
	the best approximation algorithm  for finding  $(p,v^*)\in \hilbertH\times \hilbertG$ such that
	\begin{align*}
	0\in Ap+B(Lp)  \quad \text{and}\quad 0\in -LA^{-1}(-Lv^*)+B^{-1}v^*
	\end{align*}
	(see \cite{Alotaibi_2015_best}). Let us recall that for any $\gamma>0$, ${}^\gamma A:\hilbertH\to\hilbertH$ is Yosida approximation of $A$,  ${}^\gamma A  = \frac{1}{\gamma} (Id - J_{\gamma}A)$. 
\end{enumerate}

For other multifunctions $\multifC$ and other properties of projections onto moving sets, see, e.g., \cite[Theorem 3.1]{MR935276},  \cite[Theorem 3.10]{MR358181}, \cite[Theorem 2.1]{MR1354777},  \cite[Proposition 5.2]{full_lipschitzian_and_holderian_stability_Mordukhovich} and   \cite[Example 6.4]{full_stability_of_general_parametric_variational_systems}.\\[1cm]

\bibliographystyle{plain}      

\bibliography{bibliography_ENG}   

\section{Appendix - auxiliary Lemmas}

\begin{remark}\label{remark:setD_ball_bounded}
	Fact \ref{fact:subset_ball} implies that $\setD\subset \bar{\mathbb{B}}(\frac{\bar{w}+\bar{z}}{2},\frac{\|\bar{w}-\bar{z}\|}{2})$, hence $\setD$ is bounded. Moreover, this easily implies
	\begin{equation*}
	\|\bar{w}-\bar{z}\|=d:=\sup_{x,y\in \setD } \|y-x\|,
	\end{equation*}
	that is, the pair $\bar{z}$, $\bar{w}$ realizes maximal distance between two points in $\setD$ (the diameter of $\setD$).
\end{remark}

\begin{lemma}\label{lemma:continuous}
	Let $x(\cdot)\in C([a,b],\hat{\setD})$, $[a,b]\subset \mathbb{R}_{+}$. Then the function $f(t):=F(x(t))$ is continuous: $f(\cdot)\in C([a,b],\hilbertX)$
\end{lemma}

\begin{lemma}(about extendibility
	)\label{lemma:prolongation}
	Let $x(t)$ be defined and differentiable in a continuous way in left-sided neighbourhood of $t_0$, i.e.
	\begin{equation}\label{eq_9}
	x(\cdot)\in C^1((t_0-\gamma,t_0),\hat{\setD})
	\end{equation}
	and assume that the limit
	\begin{equation}\label{eq_10}
	x_1:=\lim_{t\rightarrow t_0{-}} \dot{x}(t)
	\end{equation}
	exists and $x_1\in F(\hat{\setD})$. Then
	\begin{enumerate}
		\item $x(t)$ is extendable in a continuous way to function $\tilde{x}(\cdot)\in  C^1((t_0-\gamma,t_0],\hat{\setD}) $
		\item $\dot{\tilde{x}}_\ell = x_1$ (where $\ell$ denotes the left derivative of $x(\cdot)$ at $t_0$)
	\end{enumerate}
\end{lemma}

\begin{proof}[ of Lemma~\ref{lemma:prolongation}]
	It follows from the existence of the left-hand limit that  the derivative is bounded in some left-sided half-neighbourhood of $t_0$:
	\begin{equation} 
	\label{eq_11}
	\exists \zeta\in(0,\gamma],\ \exists L>0\ \forall t\in (t_{0}-\zeta, t_{0})\ \|\dot{x}(t)\|\le L.
	\end{equation}

	By the weakened formula for finite increments, we obtain Lipschitz continuity of the function $x(t)$ on $(t_{0}-\zeta, t_{0})$ with some constant $L$. Therefore, for the function $ x (t) $, the Cauchy condition for the existence of the left derivative at time $t_0$ is satisfied, and
	\begin{equation} 
	\label{eq_12}
	\exists\ x_{0}=\lim_{t\rightarrow t_{0}{-}} x(t).
	\end{equation}
	Put
	$$
	\tilde{x}(t)=\left\{\begin{array}{ll}
	x(t),& t\in (t_{0}-\gamma, t_{0});\\
	x_{0}& t=t_{0}.
	\end{array}
	\right.
	$$
	It is obvious that, the function constructed in this way is continuous on $(t_{0}-\zeta, t_{0}]$. Now, it is enough to show that $\dot{\tilde{x}}_\ell(t_{0})=x_{1}$, i.e. 
	$$
	\lim_{t\rightarrow t_{0}^{-}}\frac{1}{t-t_{0}}(x(t)-x_{0})=x_{1},
	$$
	or
	$$
	\lim_{\Delta t\rightarrow 0}\frac{1}{\Delta t}(x_{0}-x(t_{0}-\Delta t))=x_{1}.
	$$
	To use the formula of Newton-Leibniz 
	we introduce a function
	$$
	z(t)=\left\{\begin{array}{ll}
	x'(t),& t\in (t_{0}-\gamma, t_{0});\\
	x_{1}& t=t_{0}.
	\end{array}
	\right.
	$$
	By  \eqref{eq_9} and \eqref{eq_10}, function $z(t)$ is continuous on $(t_{0}-\zeta, t_{0}]$. We cannot yet claim that  $\dot{\tilde{x}}(t_{0})=x_{1}$, our aim is  to prove it.
	
	For any $\delta\in(0,\zeta)$ we can rewrite the formula of  Newton-Leibniz as
	\begin{equation} 
	\label{eq_13}
	\int_{t_{0}-\Delta t}^{t_{0}-\delta} z(t) dt=x(t_{0}-\delta)-x(t_{0}-\Delta t).
	\end{equation}
	
	We take the limit with $\delta $ tending to zero. Then on the one hand, $x(t_{0}-\delta)\rightarrow x_{0}$ (see \eqref{eq_12}). On the other hand,
	$$\int_{t_{0}-\Delta t}^{t_{0}-\delta} z(t) dt \rightarrow \int_{t_{0}-\Delta t}^{t_{0}} z(t)\ dt,
	$$
	because
	$$
	\|\int_{t_{0}-\Delta t}^{t_{0}} z(t) dt-\int_{t_{0}-\Delta t}^{t_{0}-\delta} z(t) dt\|=\|\int_{t_{0}-\delta}^{t_{0}} z(t) dt\|\le \int_{t_{0}-\delta}^{t_{0}}\|z(t)\| dt\le \delta L\rightarrow 0.
	$$
	Here, we used continuity of $z(t)$, estimation \eqref{eq_11} and, from fact \eqref{eq_10} with \eqref{eq_10}, estimation $\|z(t_{0})\|\le L$.
	
	Taking the limit on both sides in \eqref{eq_13} we obtain
	\begin{equation} 
	\label{eq_14}
	\int_{t_{0}-\Delta t}^{t_{0}} z(t) dt = x_{0}-x(t_{0}-\Delta t).
	\end{equation}
	Then from \eqref{eq_13} and \eqref{eq_14} we have
	$$
	\int_{t_{0}-\Delta t}^{t}-\tilde{x}(t)-\tilde(t_{0}-\Delta t)
	$$
	for all $t\in[t_{0}-\Delta t, t_{0}]$, where the function inside the integral is continuous. 
	Now applying at the point $t = t_0$ the theorem on differentiation of the integral with respect to the upper limit, we obtain  
	$$
	\dot{\tilde{x}}_{\ell}(t_{0})=z(t_{0})=x_{1},
	$$ 
	as required.  
\end{proof}

\begin{lemma}\label{Lemma:limit_Cauchy}
	Let $x(t)$ be a Lipschtiz function on $(a,b)$, $a,b\in \mathbb{R}$ with Lipschitz constant $L$ and values in Hilbert space $\hilbertX$. Then the limit $\lim_{t\rightarrow b^-} x(t)$ exists.  
\end{lemma}
\begin{proof}
	It is enough to show that $x(t)$ has the Cauchy property at $b^-$, in the sense that
	\begin{align}\label{condition:Cauchy}
	\begin{aligned}
	&\forall \varepsilon>0\ \exists \delta>0 \ \forall t_1,t_2\in (a,b) \\
	& b-t_1<\delta \ \wedge b-t_2<\delta \implies \|x(t_1)-x(t_2)\|<\varepsilon.
	\end{aligned}
	\end{align}
	Since $x(t)$ is Lipschitz on $(a,b)$ we have
	\begin{equation*}
	\forall t_1,t_2 \in (a,b) \quad \|x(t_1)-x(t_2)\|< L |t_1-t_2|
	\end{equation*}
	Let us take any $\varepsilon >0$ and $\delta = \frac{\varepsilon}{2L}$. Then for any $t_1,t_2$, $0<b-t_1<\delta$, $0<b-t_2<\delta$ we have
	\begin{equation*}
	\|x(t_1)-x(t_2)|< L |t_1-t_2|< L (|b-t_1|+|b-t_2|)< \varepsilon,
	\end{equation*}
	which proves \eqref{condition:Cauchy}.  
\end{proof}
\begin{lemma}\label{inequaltiy:triangle_ball}	
	For all $x\in \setD$ we have
	\begin{equation*}
	\|x-\bar{z}\|^2\leq \|\bar{w}-\bar{z}\|^2 - \|\bar{w}-x\|^2.
	\end{equation*}
\end{lemma}
\begin{proof}
	This  follows from \eqref{inequality:setD}:  we have
	\begin{equation*}
	\|\bar{w}-\bar{z}\|^2=\|\bar{w}-x\|^2+2\langle \bar{w}-x \mid x-\bar{z} \rangle + \|x-\bar{z}\|^2\geq \|\bar{w}-x\|^2+ \|x-\bar{z}\|^2,
	\end{equation*}
	for all $x\in \setD$.   
\end{proof}

\end{document}